\documentclass[12pt,reqno]{amsart}
\usepackage{amsmath,amsfonts,color,amsthm, amssymb,amsbsy}
\usepackage{enumitem}
\usepackage{soul}
\usepackage{multicol}
\usepackage[dvips]{graphicx}
\allowdisplaybreaks

\numberwithin{equation}{section}
\makeatletter
\newcommand{\rom}[1]{\romannumeral #1}
\newcommand{\Rom}[1]{\expandafter\@slowromancap\romannumeral #1@}
\makeatother

\newtheorem{thm}{Theorem}

\newtheorem{prop}[thm]{Proposition}
\newtheorem{lemma}[thm]{Lemma}
\newtheorem{corr}[thm]{Corollary}

\newtheorem{rmk}{Remark}
\newtheorem*{ex}{Example}
\newtheorem*{no}{Notation}

\newcommand\dis{\displaystyle}
\def\p{\partial}

\def\bf{\boldsymbol}
\newcommand\sk{\medskip\noindent}
\def\ep{\varepsilon}
\def\O{\mathcal{O}}

\begin{document}

\title[Blow-up solutions to the REP equations]{Necessary conditions for blow-up solutions to the restricted Euler--Poisson equations}
\author{Hailiang Liu and Jaemin Shin}
\address{Department of Mathematics, Iowa State University, Ames, IA 50011, United States} \email{hliu@iastate.edu}
\address{Department of Mathematical Sciences \& Institute for Applied Mathematics and Optics, Hanbat National University, Daejeon 34158, Korea}\email{jaemin.shin@hanbat.ac.kr}


\begin{abstract}
In this work, we study the behavior of blow-up solutions to the multidimensional restricted Euler--Poisson equations
which are the localized version of  the full Euler--Poisson system. We provide necessary conditions for the existence of finite-time blow-up solutions in terms of the initial data, and describe the asymptotic behavior of the solutions near blow up times. We also identify a rich set of the initial data which yields global bounded solutions.
\end{abstract}

\maketitle

\section{Introduction and main results}
\sk In this paper, we consider the following ordinary differential equation (ODE) system
\begin{subequations}\label{main}
\begin{align}
&\lambda_i ' = -\lambda_i^2 + \frac{k}{n}(\rho-c_b), \quad i = 1, 2, \cdots, n, \quad t>0, \label{lambda}\\
& \rho' = -\rho \lambda, \quad \lambda = \sum_{i=1}^n \lambda_i, \label{rho}\\
& \rho(0)=\rho_0>0,  \quad \lambda_i(0)=\lambda_{i,0},
\end{align}
where ${}'$ is the derivative in time $t$, $k, c_b$ are positive parameters, and $n\geq 2$ is an integer.  This system
 proposed in \cite{LiuTad02} is a localized version of the Euler--Poisson equations, hence called the restricted Euler--Poisson (REP) system in the literature. We assume that the initial data for $\lambda_i$ are real and satisfy the order condition
\begin{equation}\label{J}
\lambda_{1,0}=\cdots = \lambda_{J,0} < \lambda_{J+1, 0} \leq \cdots \leq \lambda_{n,0}.
\end{equation}
\end{subequations}
Here, we introduce a quantity $1 \leq J \leq n$ with which we characterize
the number of the initial $\lambda_{i}$ coinciding with $\lambda_{1, 0}$.
The order of $\lambda_i$'s is known to be preserved (see \cite{LiuTad02, LeeLiu13} and Lemma \ref{lem:oreder}).  The purpose of this work is to identify necessary conditions for the existence of blow-up solutions to this REP system, and study the detailed solution behavior near the blow-up time.

\sk
To understand the physical meaning of each term, we recall the full Euler--Poisson equations for the velocity field $\bf{u}$ and local density $\rho$,
\begin{subequations}\label{EP}
\begin{align}
&\bf{u}_t +  \bf{u}\cdot \nabla \bf{u} = k \nabla \Delta^{-1} (\rho -c_b),\; x\in \mathbb{R}^n, \; t>0, \\
& \rho_t + \nabla \cdot(\rho \bf{u}) = 0,
\end{align}
\end{subequations}
where the constant $k$ represents a repulsive ($k > 0$) or attractive ($k < 0$) force, and $c_b$ denotes the background state. This system \eqref{EP} describes the dynamic behavior of several important physical flows, including those for semi-conductors, plasma physics, and the collapse of stars (see \cite{HJL81, M86, MRS90, BRR94, BW98, DLYTY02}). The existence and behaviors of solutions for \eqref{EP} and related problems have been extensively studied under various assumptions;  see, e.g., \cite{JR00, ELT01, CT08, WTB12, CCTT16} and references therein. In particular, in \cite{LiuTad02} Liu and Tadmor introduced the method of spectral dynamics, which serves as a powerful tool to study dynamics of the velocity gradient $M=\nabla \bf{u}$ along particle paths. Indeed, \eqref{EP}
can be converted into
\begin{align*}
&M' = - M^2 + k \mathcal{R}[\rho -c_b],\\
&\rho' = - \rho \text{tr}M,
\end{align*}
where $'$ is the convective derivative, $\p_t + \bf{u}\cdot \nabla$,
and $\mathcal{R}$ is the Riesz matrix operator,
\begin{equation*}
\mathcal{R}[f]:= \nabla \otimes \nabla \Delta^{-1}[f].
\end{equation*}
It is the global nature of the Riesz matrix, $\mathcal{R}[\rho -c_b]$,
which makes the issue of regularity for Euler--Poisson equations such an intricate question to solve, both analytically and numerically.  In this paper we focus on the REP equation for $M$ which was proposed
in  \cite{LiuTad02} by restricting attention to the local isotropic trace, $\frac{k}{n}(\rho-c_b)
I_{n\times n}$, of the global coupling term $k\mathcal{R}[\rho-c_b]$ , i.e.,
\begin{subequations}\label{REP_M}
\begin{align}
&M' = - M^2 + \frac{k}{n} (\rho - c_b) I_{n\times n},\\
&\rho' = - \rho \text{tr}M.
\end{align}
\end{subequations}
This is a matrix Ricatti equation for the $n\times n$ matrix $M$, coupled with the density equation,
which should mimic the dynamics of $(\rho, \nabla \bf{u})$ in the full Euler--Poisson equations.
The REP system \eqref{main} for the eigenvalue $\lambda_i$ of  $M$ follows from \eqref{REP_M}.
We note that the REP system \cite{LiuTad02} is to the full Euler--Poisson equations what the restricted Euler
(RE) model is to the full Euler equations, while the RE system is known useful in understanding 
the local topology of the Euler dynamics; we refer the reader to \cite{Vi82, BP95, Can92, CPS98}.

\sk The existence of a critical threshold phenomenon associated with this 2D REP
model with zero background, $c_b=0$, was first identified in \cite{LiuTad02}.  A precise description of the critical threshold for the 2D REP system, with both zero and nonzero background charges, was given in \cite{LiuTad03}. These results have been extended to  multi-dimensional REP equations by Lee and Liu in \cite{LeeLiu13}. While  Lee identified upper-thresholds for finite time blow-up solutions to an improved REP equation in two dimensions (\cite{Lee17}). It is worth mentioning that critical thresholds for restricted Euler equations were studied in \cite{LTW10} and \cite{Wei11}.

\sk In this work, we attempt to advance our understanding of the critical threshold phenomenon by providing necessary conditions for the existence of finite-time blow-up solutions to the REP system \eqref{main}.
Our results thus provide a complement to the existing results in \cite{LiuTad02, LiuTad03, Lee17} for REP systems.

\sk  In order to see the subtleness of the problem, we 
recall that a movable essential singularity cannot be achieved for a first-order scalar differential equation $u' = F(t, u)$, as long as $F$ is a rational function of $u$ with coefficients that are algebraic functions of $t$ (\cite{Hill97}). However, this is not the case for the system of equations considered here. In other words, the singularity types of solutions $\lambda$ and $\rho$ of \eqref{main} are not known a priori. This is one of the main difficulties with this problem, because we cannot simply utilize some balance equations to analyze the behavior of solutions near a singular point. To overcome this difficulty, we transform the Riccati-type equations \eqref{lambda} 
into second-order linear differential equations for
\begin{equation*}
u_i(t) = e^{\int_0^t \lambda_i(s) ds}.
\end{equation*}
By analyzing the general solution to the second-order differential equation, we are able to reveal the behavior of $u_i$, which also provides information on the behavior of $\lambda_i'$. Indeed, we can characterize the asymptotic behaviors of $\lambda_i$ and $\rho$ near the blow-up time by the gap of the initial data $\lambda_{i,0}$ together with $\rho_0$.

\sk The quantity $J$ defined in \eqref{J} is critical in terms of different solution behaviors of $\lambda_i$.
We state our main results in the following.  

\begin{thm}\label{mainthm} Suppose that the maximum interval of existence for \eqref{main} 
is $[0,t_B)$ for some $0<t_B<\infty$. Then
\begin{equation}\label{main:J}
1 \leq J \leq \frac n2 \quad \text{and}
\end{equation}
\begin{equation}\label{estimation_tB}
   t_B \geq \frac{1}{\omega} \arctan \left(\frac{\lambda_{1,0}}{\omega}\right) + \frac{\pi}{2\omega}, \quad \omega:=\sqrt{kc_b/n}.
   \end{equation}
Moreover,
\begin{subequations}\label{lambda_inf}
\begin{align}
\lim_{t\to t_B^-} \lambda_{i}(t) &= -\infty,  \quad 1 \leq i \leq J, \\  \lim_{t\to t_B^-} \lambda_{i}(t) &= \infty, \quad  J < i \leq n,\\
\lim_{t\to t_B^-} \rho(t) &= \infty,
\end{align}
\end{subequations}
and also
\begin{subequations}\label{keylemma}
\begin{eqnarray}
&\dis\lim_{t\to t_B^-} \lambda_1(t) e^{\int_0^t \lambda_1(s)+ \lambda_n(s) ds } = -p,\\
&\dis\lim_{t\to t_B^-} \lambda_n(t) e^{\int_0^t \lambda_1(s)+ \lambda_n(s) ds } = q,
\end{eqnarray}
\end{subequations}
for some $0 \leq q \leq p$.
\end{thm}

\begin{rmk}
An interesting feature of the behavior of $\lambda_i$ is that $\lambda_i$ diverges to $-\infty$ if and only if $\lambda_{i,0} = \lambda_{1,0}$ and all the other $\lambda_i$ diverges to $+\infty$. Moreover, $J$ cannot exceed $n/2$ and in the case of $J \geq 3$, $J$ is strictly smaller than $n/2$; see Theorem \ref{mainthm+}. The limits in \eqref{keylemma} indicate how $\lambda_i$ are connected  through
$p$ and $q$, which are mainly characterized by the gap of the initial data $\lambda_{i,0}$ together with $\rho_0$.
\end{rmk}
\noindent 
Our second theorem gives the detailed blow-up rates of solutions. We note that $\lambda_i(t) = \lambda_1(t)$ for $1 \leq i\leq J$ (see Lemma \ref{lem:oreder}).  

\begin{thm}\label{mainthm+} Under the hypothesis in Theorem \ref{mainthm},  the blow-up rates of singular solutions depend on the size of $J$, and can be made more precise as follows:
\begin{itemize}
\item[(\rom{1})] If $J=1$, then $n\geq 2$ and 
\begin{align*}
& \dis \lim_{t\to t_B^-} (t_B-t) \lambda_1 (t) =-1,\\
& \dis \lambda_i (t) = \O\left(|\ln (t_B-t)|\right) ~\text{as~} t\to t_B^-, \quad 2 \leq i \leq n,\\
& \dis \rho(t) = \O \Big(\frac{1}{t_B-t} \Big) ~ \text{as~} t\to  t_B^-.
\end{align*}

\item[(\rom{2})] If $J=2$, then  $n\geq 4$ and one of the following cases must hold:
\begin{itemize}
\item[(a)] 
If $n \geq 5$, then
\begin{align*}
& \dis \lim_{t\to t_B^-} (t_B-t) \lambda_1 (t) = -1,\\
& \dis \lim_{t\to t_B^-} (t_B-t) \lambda_i (t) = 0, ~\lim_{t\to t_B^-} \int_0^{t} \lambda_i(s) ds = \infty, \quad 3 \leq i \leq n,\\
& \dis \rho(t) = o \left(\frac{1}{(t_B-t)^2} \right) ~ \text{as~} t\to  t_B^-.
\end{align*}
\item[(b)] 
If $n=4$ and ${(\lambda_{1,0}-\lambda_{3,0})(\lambda_{1,0}-\lambda_{4,0})=:A_0 > {k \rho_0} }$, then
\begin{align*}
& \dis \lim_{t\to t_B^-} (t_B-t)\lambda_1 (t)  = -\frac{1}{2} -\frac{1}{2}\sqrt{\frac{A_0}{A_0 - k \rho_0}},\\
& \dis \lim_{t\to t_B^-} (t_B-t)\lambda_i (t)   = -\frac{1}{2} + \frac{1}{2}\sqrt{\frac{A_0}{A_0- k \rho_0}}, \quad i =3,4,\\
& \dis \rho(t) = \O \Big(\frac{1}{(t_B-t)^2} \Big) ~ \text{as~} t\to  t_B^-.
\end{align*}

\item[(c)] 
If $n=4$ and ${A_0 = {k \rho_0}}$, then there exists $C >0$ such that
\begin{align*}
& \dis \lim_{t\to t_B^-} (t_B -t)^2 \lambda_1 (t)  = -C,\\
& \dis \lim_{t\to t_B^-} (t_B -t)^2 \lambda_i (t)  = C, \quad i=3,4, \\
& \dis \lim_{t\to t_B^-} (t_B -t)^4 \rho(t) = \frac{k}{4}C^2.
\end{align*}
\end{itemize}
\item[(\rom{3})]
If $J\geq 3$, then $n > 2J$ and there exists $C > 1$ such that
\begin{align*}
& \dis \lim_{t\to t_B^-} (t_B-t)\lambda_1 (t)  = -C,  \\
& \dis \lim_{t\to t_B^-} (t_B-t)\lambda_i (t)  = C -1, \quad J+1 \leq i \leq n,\\
& \dis \rho(t) = \O \Big(\frac{1}{(t_B-t)^{2}} \Big) ~ \text{as~} t\to  t_B^-.
\end{align*}

\end{itemize}
\end{thm}
\begin{rmk} Note that for each $J$ specified in different cases, $n$ has to be in certain range so to  fulfill the
requirement that the maximum interval of existence for \eqref{main} be finite. 
\end{rmk}

\sk In contrast to the finite-time breakdown, the multi-dimensional REP equations admit a large class of global bounded solutions. Our  results are summarized below.
\begin{thm}\label{main3}
If
\begin{equation*}
J > \frac{n}{2},
\end{equation*}
or
\begin{equation*}
J \geq 3\quad \text{and} \quad J=\frac{n}{2},
\end{equation*}
then \eqref{main} 
has a global bounded solution.
\end{thm}
\noindent This has improved upon some global existence results in  \cite{LeeLiu13}, in particular Theorem 2.3 (corresponding to $n=3, J=3$)
and Theorem 2.9 (corresponding to $J=n$) therein.

\sk The remainder of this paper is organized as follows. In Section 2, we first show that no highly oscillating solution exists (see \eqref{def2} for the definition of a highly oscillating solution). That is, we show that  $\rho$ and $|\lambda_i|$ diverge to $\infty$ for some $i$ when \eqref{main} admits a finite-time blow-up solution. We also provide a proof of \eqref{estimation_tB}. In Section 3, we transform \eqref{lambda} to a second-order linear differential equation, and demonstrate the solution behaviors of \eqref{keylemma}. To prove Theorem \ref{mainthm+} using \eqref{keylemma}, we consider the subcases 
\begin{equation*}
p>q \quad\text{and}\quad p=q.
\end{equation*}
We study the case with $p>q$ in Section 4. Here, we show that the coefficients of leading singular order terms for $\lambda_1$ and $\lambda_n$ can be represented as $-p/(p-q)$ and $q/(p-q)$, respectively. We also conclude that this case yields  (\rom{1}), (a) and (b) of (\rom{2}), or (\rom{3})
in Theorem \ref{mainthm+}. The last section deals with the case where $p=q$,  which implies
(c) of (\rom{2}) in Theorem \ref{mainthm+}. The main difficulty in this case lies in that the leading singular terms of $-\lambda_1$ and $\lambda_n$ are the same. {For this reason}, we have examined the second singular terms. We also provide explicit solutions to the REP system \eqref{main} assuming that $\lambda_{3,0}=\lambda_{4,0}$. We remark that \eqref{main:J} and \eqref{lambda_inf} follows from Theorem \ref{mainthm+}.   

\begin{no}
Throughout the paper we write
\begin{equation*}
f(x) = \O (g(x))\quad \text{as}~ x \to x_0^-,
\end{equation*}
if there are $M, \delta >0$ such that
\begin{equation*}
|f(x)| \leq M |g(x)| \quad \text{for~all~} x_0 - x < \delta.
\end{equation*}
Similarly, we write
\begin{equation*}
f(x) = o(g(x)) \quad \text{as}~ x \to x_0^-,
\end{equation*}
if for any $\ep>0$ there is $\delta>0$ such that
\begin{equation*}
|f(x)| \leq \ep |g(x)|\quad \text{for~all~} x_0 - x < \delta.
\end{equation*}
\end{no}

\section{Non-oscillating solutions}

\sk Suppose that $[0,t_B)$ is the maximum interval of existence of solutions to an ordinary differential equation (ODE) $ u' = F(t, u)$. Then, either
\begin{equation}\label{def1}
\lim_{t\to t^-_B}| u(t)| = \infty,
\end{equation}
or
\begin{equation}\label{def2}
0< \limsup_{t\to t^-_B} u(t) - \liminf_{t\to t^-_B} u(t).
\end{equation}
Here, we define $\infty-\infty =0$. We say that a solution blows up at a finite time if it satisfies \eqref{def1}, and is oscillating at a finite time if it satisfies \eqref{def2}. Note that $u(t)$ satisfying
\eqref{def1} may be oscillating in the standard sense, i.e., $u(t) \to \infty$ but $u'(t) \ngtr 0$ (or $u(t) \to - \infty$ but $u'(t) \nless 0$).

\sk For the REP system (\ref{main}), we can prove that there exist no finite-time oscillating solutions. More precisely, the following proposition holds.
\begin{prop}\label{thm:nonosc}
Suppose that the maximum interval of existence for \eqref{main} is $[0,t_B)$ for some $0<t_B<\infty$. Then, it holds for some $i$ that
\begin{equation}\label{abs(lambda)infty}
\lim_{t\to t^-_B} |\lambda_i(t)| = \infty.
\end{equation}
\end{prop}

\begin{proof}

If $\lambda_i$ is assumed to be an oscillating solution of type \eqref{def2}, then there exists a sequence of disjoint intervals $(a_m, b_m) \subset (0,t_B)$ on which $\lambda_i$ is decreasing and
\begin{eqnarray}
&\dis\lim_{m\to \infty} (b_m - a_m)=0, \label{new3.2}\\
&\dis\lim_{m\to \infty} ( \lambda_i(b_m) - \lambda_i(a_m) )  <0, \label{new3.3}\\
&\dis\lim_{m\to \infty} ( |\lambda_i(b_m)| + |\lambda_i(a_m)| )  < \infty. \label{new3.1}
\end{eqnarray}
Note that if \eqref{new3.1} fails, then one can conclude that $\lim_{t \to t_B} |\lambda_i(t)|= \infty$, although $\lambda_i$ is oscillating in the standard sense.

\sk From \eqref{rho} 
\begin{equation*}
\rho = \rho_0 e^{-\int_0^t \lambda(s) ds} >0.
\end{equation*}
It follows from \eqref{lambda} that 
\begin{equation*}
\lambda'_i \geq - \lambda_i^2 - \omega^2, \quad \omega:= \sqrt{\frac{kc_b}{n}}.
\end{equation*}
That is,
\begin{equation}\label{add1}
\frac{\lambda_i'}{\lambda_i^2 + \omega^2} \geq  - 1.
\end{equation}
Upon integration over $(a_m,b_m)$, this implies that
\begin{equation*}
\frac1 \omega \arctan \Big(\frac{\lambda_i(b_m)}{\omega}\Big) - \frac1\omega \arctan \Big(\frac{\lambda_i(a_m)}{ \omega}\Big) \geq  - (b_m-a_m).
\end{equation*}
Owing to the conditions \eqref{new3.3} and \eqref{new3.1}, the left-hand side is strictly less than $0$, while the right-hand side converges to $0$. Thus, there exist no oscillating solutions of type \eqref{def2} for $\lambda_i$.
\end{proof}

\sk The order-preserving property of $\lambda_i$ is well known (see \cite{LiuTad02, LeeLiu13}). Indeed, it follows from \eqref{lambda} that
\begin{equation}\label{differenc_lambda}
(\lambda_i -\lambda_j)' = - (\lambda_i + \lambda_j)(\lambda_i - \lambda_j),
\end{equation}
and this yields the following lemma.
\begin{lemma}\label{lem:oreder}
For any $t>0$, the solutions $\lambda_i$ of \eqref{main} satisfy 
\begin{equation*}
\lambda_1(t) =  \cdots = \lambda_J(t) < \lambda_{J+1}(t) \leq \cdots \leq \lambda_n(t).  
\end{equation*} 
\end{lemma}
\sk Proposition \ref{thm:nonosc} states that $\dis \lim_{t\to t_B^-} |\lambda_i(t)| =\infty$ for some $i$. Then, 
one can conclude that 
\begin{equation*}
\lim_{t\to t_B^-} \lambda_1 (t) = -\infty,
\end{equation*}
owing to Lemma \ref{lem:oreder} concerning order preservation. In fact, if we assume that there exists no
$\lambda_i$ diverging to $-\infty$, then $\lambda_j$ tends to $+ \infty$ for some $j$ as $t \to t_B^-$, so does $\lambda  = \sum_{i=1}^n \lambda_i$. Thus, 
\begin{equation*}
\dis\min_{0\leq t\leq t_B} \lambda(t) > -\infty, 
\end{equation*}
and 
\begin{align*}
\lambda_j'(t) &=  - \lambda_j^2 (t) + \frac{k \rho_0}{n} e ^ {-\int_0^t \lambda(s) ds }  - \omega^2\\
           &\leq  \frac{k \rho_0}{n} e^{- t_B \dis\min_{0\leq t \leq t_B} \lambda(t)} -\omega^2 < \infty,
\end{align*}
which contradicts the fact that $\lambda_j \to \infty$ as $t \to t_B^-$. Thus, it must hold that $\lim_{t\to t_B^-} \lambda_i(t) = -\infty$ for some $i$. Owing to the order preservation in Lemma \ref{lem:oreder}, the following proposition holds.
\begin{prop}\label{thm:lambda_J}
Suppose that the maximum interval of existence for \eqref{main} is $[0,t_B)$ for some $0<t_B<\infty$. Then, there exist $1\leq J_1 \leq J_2 \leq n$ such that
      \begin{eqnarray}
      ~\dis \lim_{t\to t_B^-}\lambda_i (t) = - \infty, ~ 1\leq i\leq J_1,\label{lambda1}\\
      \lim_{t\to t_B^-} \lambda_i(t) = \infty, \quad J_2< i \leq n. \label{j'}
      \end{eqnarray}
\end{prop}

\sk We remark that there exists no $\lambda_{i}$ satisfying \eqref{j'} in the case that $J_2=n$. However, \eqref{lambda1} indicates that 
\begin{equation*}
\lim_{t\to t^-_B} \lambda_1 = -\infty.
\end{equation*}

\sk The estimation \eqref{estimation_tB} of $t_B$ also follows immediately. 
Integrating \eqref{add1} for $i=1$ over $(0,t)$ 
yields that  
\begin{equation*}
\arctan\Big(\frac{\lambda_1(t)} {\omega}\Big) > \arctan\Big(\frac{\lambda_{1,0}}{\omega}\Big) - \omega t.
\end{equation*}
Sending $t$ to $t_B$ implies that
\begin{align*}
 -\frac\pi 2  \geq \arctan \Big(\frac{\lambda_{1,0}}{\omega}\Big) - \omega t_B.
\end{align*}
Thus, we obtain \eqref{estimation_tB} in Theorem \ref{mainthm}.
\begin{thm}
Suppose that the maximum interval of existence for \eqref{main} is $[0,t_B)$ for some $0<t_B<\infty$. Then,
\begin{equation*}
   t_B \geq \frac{1}{\omega} \arctan \left(\frac{\lambda_{1,0}}{\omega}\right) + \frac{\pi}{2\omega}.
   \end{equation*}
\end{thm}

\sk
Next, we turn our attention to density $\rho$. Here, we show that $\rho \notin L^1(0,t_B)$ through a contradiction argument. Assuming that $\rho \in L^1(0,t_B)$, we find that $J_2 = n$, because integrating \eqref{lambda} gives that for $ i = 1, 2, \cdots, n$
\begin{equation*}
\lambda_i(t) - \lambda_{i,0} = - \int_0^t \lambda_i^2(s) ds + \frac{k}{n} \int_0^t (\rho(s) -c_b) dx < \infty.
\end{equation*}
It follows that $J_1 = n$ or $\lambda_i$ is finite for $J_1 < i \leq n$. Thus, there exists $f(t)$ such that  
\begin{align}\label{temp1}
\rho(t) = \rho_0 e^{-\int_0^t \lambda(s) ds}  = f(t) e^{- \sum_{i=1}^{J_1} \int_0^t \lambda_i(s) ds}, \quad 0< f(t) < \infty.
\end{align}
Now, let $1 \leq i \leq J_1$. Then, there exists $t_1 \in (0,t_B)$ such that
\begin{align*}
\lambda_i (t) < 0, \quad t_1 < t < t_B,  
\end{align*}
and $\tilde\lambda_i := \lambda_i - \omega$ satisfies 
\begin{equation*}
\tilde\lambda_i ' = \lambda_i'  > -\lambda_i^2  - \omega^2  > - \tilde \lambda_i^2.
\end{equation*}
We then deduce that 
\begin{equation*}
\tilde\lambda_i < - \frac{1}{t_B-t}, \quad t_1 < t < t_B.
\end{equation*}
Thus, for some constant $K >0$ it holds that
\begin{equation*}
- \sum_{i=1}^{J_1} \int_0^t \lambda_i(s) ds >  \ln (K (t_B -t)^{-J_1}) , \quad t_1 < t < t_B.
\end{equation*}
Substituting the inequality into \eqref{temp1} yields 
\begin{equation*}
\rho(t) = f(t) e^{- \sum_{i=1}^{J_1} \int_0^t \lambda_i(s) ds} > \frac{K f(t) }{(t_B -t)^{J_1}}, \quad t_1 < t < t_B. 
\end{equation*}
In Proposition \ref{thm:lambda_J}, we have shown $J_1\geq 1$, which contracts to the assumption that $\rho \in L^1(0,t_B)$. We summarize this result as follows.
\begin{prop}\label{lem:rho}
Suppose that the maximum interval of existence for \eqref{main} is $[0,t_B)$ for some $0<t_B<\infty$. Then,
\begin{equation}\label{rhoinfty}
\lim_{t\to t^-_B} \int_0^t \rho(s)ds  = \infty.
\end{equation}
\end{prop}

\section{Transformed equations}

\sk Although Proposition \ref{thm:lambda_J} and Proposition \ref{lem:rho} state that for some $i$, $\lambda_i$ and $\int_0^t \rho(s) ds$ diverge as $t\to t_B^-$, respectively, they do not illuminate the behaviors of $\lambda'_i$ and $\rho$ near $t_B$, which are essential for analyzing solution singularities. To go further, we transform the Riccati-type equation \eqref{lambda} to a second-order linear differential equation by defining 
\begin{equation}\label{def_ui}
\dis  u_i(t)  = e^{ \int_0^t \lambda_i(s) ds}.
\end{equation}
This gives 
\begin{align}
&\lambda_i = \frac{u'_i}{u_i},\nonumber\\
& u_i(0)=1, \quad u'_i(0)=\lambda_{i,0}, \label{u_init}
\end{align}
and
\begin{align}\label{rho_u}
\rho(t)  = \rho_0 e^{- \sum_{i=1}^n \int_0^t \lambda_i(s) ds} = \rho_0  \prod_{i=1}^n \frac{1}{u_i(t)}.
\end{align}
Equation \eqref{lambda} is also transformed to
\begin{equation}\label{eqn u1}
u_i'' - \frac{k}{n}(\rho-c_b) u_i =0,
\end{equation}
or (recall that $\omega^2 = kc_b/n$) 
\begin{equation}\label{eqn_u}
u''_i + \omega^2 u_i = \frac{k}{n}  \rho_0  \prod_{\substack {{m=1} \\ {m \neq i}}}^n \frac{1}{u_m} =: g_i.
\end{equation}
The general solution of \eqref{eqn_u} is thus given by
\begin{equation}\label{general_u1}
u_i(t) = c_1 \sin \omega t + c_2 \cos \omega t + \frac{1}{\omega} \int_0^t g_i(s) \sin \omega (t-s) ds.
\end{equation}
\sk We proceed to observe the behavior of $u_i$ near $t_B$. Let $1 \leq i \leq J_1 $, i.e., 
\begin{equation}\label{3.12}
 \lambda_i = \frac{u_i'}{ u_i} \to -\infty \quad \text{as~} t \to t^-_B, 
\end{equation} 
then the positivity of $u_i$ implies that $u'_i <0$, and thus $u_i$ converges. Let $- \alpha_i \leq 0$ be the least upper bound of $u'_i$. Then, for any $\ep >0$ there exists $ t_0 \in (0, t_B)$ such that 
\begin{equation*}
-\alpha_i -\ep < u'_i(t_0) \leq -\alpha_i.
\end{equation*} 
On the other hand, it follows from \eqref{eqn u1} that for any $t_0<t<t_B$,
\begin{equation}\label{pf3.1}
u_i'(t) - u_i'(t_0) = \int_{t_0}^t \frac{k}{n} (\rho(s) -c_b) u_i(s) ds. 
\end{equation}
Owing to Proposition \ref{lem:rho} and the convergence of $u_i$, there exists $t_1 \in (t_0, t_B)$ such that 
\begin{equation*}
u'_i(t) - u'_i(t_0) = \int_{t_0}^t \frac{k}{n} (\rho(s) -c_b) u_i(s) ds > 0, \quad t_1 \leq t < t_B
\end{equation*}
and
\begin{equation*}
-\alpha_i -\ep < u'_i(t_0) < u_i'(t)  \leq -\alpha_i, \quad t_0 \leq t_1 \leq t < t_B.
\end{equation*}
This implies that 
\begin{equation}\label{alpha_i}
\lim_{t\to t_B^-} u_i'(t) = -\alpha_i,
\end{equation}
and thus $u_i$ converges to $0$ as $t$ approaches $t_B$, satisfying \eqref{3.12}. We may extend the interval of existence and obtain the boundary conditions:
\begin{equation} \label{u_bdy}
u_i(t_B)=0, \quad u'_i(t_B)= -\alpha_i.
\end{equation}
It follows from (\ref{general_u1}) that
\begin{equation}\label{general_u2}
u_i(t) = \frac{1}{\omega}\left(\int^t_{t_B} g_i(s) \sin \omega (t-s) ds - \alpha_i \sin \omega(t-t_B)\right).
\end{equation}
In \eqref{pf3.1}, we observe that  
\begin{equation}\label{g_1_L1}
 g_i =\frac{k}n \rho_0 \prod_{\substack {{m=1} \\ {m \neq i}}}^n \frac{1}{u_m}  = \dis \frac{k}{n} \rho u_i \in L^1(0, t_B).
\end{equation} 

\sk Next we consider the case that $J_2 < j \leq n$;  that is, 
\begin{equation}\label{3.11}
\lambda_j = \frac{u_j'}{u_j} \to \infty.
\end{equation}
Because $u_j >0$, $u_j'$ must be positive in a neighborhood of $t_B$, which implies that $u_j$ is increasing near $t_B$. Thus,
\begin{equation}\label{beta_i}
\text{either~}  \lim_{t\to t_B^-} u_j(t) = \infty \quad \text{or~}  \lim_{t\to t_B^-} u_j(t) = \beta_j 
\end{equation}
for some $\beta_j >0$.  In either case, $u_j'$ must diverge to $\infty$, owing to \eqref{3.11}.

\sk Thanks to the behavior of $u_i$ near $t_B$, we obtain the following lemma. 
\begin{lemma}\label{lem_3}
Suppose that the maximum interval of existence for \eqref{main} is $[0,t_B)$ for some $0<t_B<\infty$. Then, for $1\leq J_1\leq J_2 \leq n$ defined in Proposition \ref{thm:lambda_J}
\begin{align*}
\lambda_{i,0} &= \lambda_{j,0}, \quad 1\leq i,j \leq J_1,\\
\lambda_i(t) &= \lambda_j(t), \quad 1\leq i,j \leq J_1,
\end{align*}
and   
\begin{align*}
\lim_{t\to t^-_B} \frac{\lambda_i(t)}{\lambda_j(t)} & =1, \quad J_2 < i,j \leq n,\nonumber\\
\lim_{t\to t_B^-}\frac{u_j(t)}{u_n(t)} & = \frac{\lambda_{j,0}-\lambda_{1,0}}{\lambda_{n,0}-\lambda_{1,0}}, \quad J_2 < j< n.
\end{align*}
\end{lemma}

\begin{proof}
We employ Abel's identity for \eqref{eqn u1} together with the initial conditions \eqref{u_init} to obtain  
\begin{equation} \label{abel_Thm}
u'_i(t) u_j(t) - u_i(t) u'_j(t)  = \lambda_{i,0} - \lambda_{j,0}, \quad 0\leq t < t_B. 
\end{equation}
Let $1\leq i,j \leq J_1$. Because $u_i(t_B)=u_j(t_B)=0$ and $u_i'(t_B), u_j'(t_B)$ are bounded, the left-hand side of \eqref{abel_Thm} vanishes at $t=t_B$, and thus $\lambda_{i,0} =\lambda_{j,0}$, as desired.

\sk We rewrite \eqref{abel_Thm} as   
\begin{equation}\label{u1un}
\lambda_i(t) - \lambda_j(t) = \frac{u_i'(t)}{u_i(t)} - \frac{u_j'(t)}{u_j(t)}   = \frac{\lambda_{i,0}-\lambda_{j,0}}{u_i(t) u_j(t)}.
\end{equation}
This yields  
\begin{equation*}
\lim_{t\to t^-_B} \frac{\lambda_i(t)}{\lambda_j(t)} =1, \quad J_2< i,j \leq n,
\end{equation*}
because  $1/{(u_i u_j)}$
converges for $J_2< i,j \leq n$.

\sk From \eqref{abel_Thm} we observe that $u_1$ and $u_n$ are linearly independent solutions of \eqref{eqn u1}. Then, for $J_2 < j < n$ we can represent $u_j$ as a linear combination of $u_1$ and $u_n$. Further  using the initial conditions \eqref{u_init}, we obtain 
\begin{equation}\label{temp61}
u_j = \frac{\lambda_{n,0}-\lambda_{j,0}}{\lambda_{n,0}-\lambda_{1,0}} u_1 + \frac{\lambda_{j,0}-\lambda_{1,0}}{\lambda_{n,0}-\lambda_{1,0}}u_n.
\end{equation}
On the other hand, the behaviors of $u_1$ and $u_j$ near $t_B$ in  \eqref{u_bdy} and \eqref{beta_i} imply that
$$\lim_{t\to t_B^-} u_1(t)/u_j(t) =0,$$ 
and it follows that      
\begin{align*}
\lim_{t\to t_B^-}\frac{u_j(t)}{u_n(t)}  = \frac{\lambda_{j,0}-\lambda_{1,0}}{\lambda_{n,0}-\lambda_{1,0}}, \quad J_2 < j< n.
\end{align*}
\end{proof}

\sk 
Further, we are able to show that $J_1=J_2=J$. That is, there exists no bounded $\lambda_i$.
More precisely, we have 

\begin{thm}\label{prop10} Suppose that the maximum interval of existence for \eqref{main} is $[0,t_B)$ for some $0<t_B<\infty$. 
Then,
\begin{equation*}
1 \leq J < n
\end{equation*}
and
\begin{equation*}
\dis \lim_{t\to t_B^-} \lambda_{i}(t) = 
\begin{cases}
-\infty,  \quad & 1\leq i \leq J,  \\
\infty,   \quad & J < i \leq n. 
 \end{cases}
\end{equation*}  
\end{thm}

\begin{proof} From Lemma \ref{lem_3} it follows that $J = J_1$. 
And also $ J = J_1 < n$; otherwise, all $\lambda_{i,0}$ would be identical, and this implies the existence of a global solution (see Theorem 2.9 in \cite{LeeLiu13}). 

\sk Now we show $J_1 = J_2$ by a contradiction argument. Indeed, if it is assumed that $J_1 < J_2$, then there exists  $|\lambda_i| <\infty$ for $J_1 <i \leq J_2$. It follows that for all $0 < t < t_B$,  
\begin{equation*}
\int_0^t \rho(s)ds  = \int_0^t \left[c_b+\frac{n}{k}(\lambda_i'(s) +\lambda_i^2(s)) \right] ds < \infty, 
\end{equation*}
which contradicts Proposition \ref{lem:rho}. 
\end{proof}

\sk Theorem \ref{prop10} implies that for $i = 1, \cdots, J$ and $j = J+1, \cdots, n$, $u_i' = u_1' < 0 $ and $u_j'>0 $ in a neighborhood of $t_B$. Because $u_1, u_j > 0$, we observe from \eqref{abel_Thm} that $u_1'u_j$ and $u_1u_j'$ should be bounded in $[0, t_B]$. Furthermore, it follows from \eqref{u1un} that $u_1u_j$ converges to $0$. 
\begin{corr}\label{cor_u1un}
Let $t_B$ and $J$ be as in Theorem \ref{prop10}, and $u_j$ as in \eqref{def_ui}. Then, for any $J < j \leq n$,
\begin{align*}
|u_1'(t) u_j(t)| < \infty, \quad 0 \leq t <  t_B, \\
|u_1(t) u_j'(t)| < \infty,  \quad 0 \leq t < t_B, 
\end{align*}
and
\begin{equation}\label{u1un0}
\lim_{t\to t_B^-} (u_1u_j)(t) = 0.
\end{equation}
\end{corr}

\sk Now, we may divide \eqref{u_bdy} and \eqref{beta_i} into the following cases, assuming that $J < j \leq n$: 
\begin{align}
&u_1'(t_B) = -\alpha_1 <0 \quad \text{and~}\dis \lim_{t\to t_B^-} u_j(t) =  \infty, \label{t_case3}\\
&u_1'(t_B) = 0 \quad \text{and~} \dis \lim_{t\to t_B^-} u_j(t) =  \beta_j>0,\label{t_case4}\\
&u_1'(t_B) = -\alpha_1 <0  \quad\text{and~} \dis \lim_{t\to t_B^-}u_j(t) = \beta_j >0, \nonumber\\
&u_1'(t_B) = 0 \quad \text{and~} \dis \lim_{t\to t_B^-} u_j(t) =  \infty. \nonumber
\end{align} 
However, \eqref{t_case3} and \eqref{t_case4} cannot occur. Indeed, \eqref{t_case3} contradicts the boundedness of $u_1' u_j$ in Corollary \ref{cor_u1un}. If \eqref{t_case4} is assumed, then $u_1' u_j \to 0$, and thus $u_1u_j' \to - \lambda_{1,0} + \lambda_{j,0}  >0$ as $t$ approaches $t_B$. It follows that, in a neighborhood of $t_B$,
\begin{equation*}
(u_1u_j)' > 0.
\end{equation*}
This also contradicts Corollary \ref{cor_u1un}, owing to \eqref{u1un0} and the fact that  $u_1u_j >0$. Thus, we have the following proposition.
\begin{prop}\label{prop:case12}
Suppose that the maximum interval of existence for \eqref{main} is $[0,t_B)$ for some $0<t_B<\infty$. 
Define $u_j$ as \eqref{def_ui}. If $ J < j \leq n$, then either
\begin{equation}\label{t_case1}
u_1'(t_B) = -\alpha_1 \text{~and~} \dis \lim_{t\to t_B^-}u_j(t) = \beta_j ~ \text{for~some~} \alpha_1, \beta_j >0,
\end{equation}
or 
\begin{equation}\label{t_case2}
u_1'(t_B) = 0 \text{~and~} \dis \lim_{t\to t_B^-} u_j(t) =  \infty 
\end{equation}
must hold. 
\end{prop}

\sk Next, we demonstrate the convergence of $u'_1u_j$ and $u_1u_j'$ for $J < j \leq n$. If  \eqref{t_case1} holds in Proposition \ref{prop:case12}, then the convergence follows from \eqref{abel_Thm}. In the case of \eqref{t_case2}, we show the convergence through several lemmas.
\begin{lemma}\label{lemma12}
Under the hypothesis of Proposition \ref{prop:case12}, for any $t \in [0,t_B)$
\begin{equation*}
 \left|\int_0^{t}   u_1'(s) u_j'(s) + g_{1j}(s) ds \right| < \infty,
\end{equation*} 
where 
\begin{equation*}
g_{ij}:=\frac{k}{n} \rho u_i u_j.
\end{equation*} 
\end{lemma}
\begin{proof}
From \eqref{eqn_u}, we deduce that
\begin{equation*}
u_1''u_j + \omega^2 u_1u_j = g_{1j}.
\end{equation*}
Integrating this equation over $[0,t]$ yields
\begin{equation}\label{g1n_1}
u_1'(t)u_j(t) - \lambda_{1,0} + \int_0^t \omega^2 u_1(s) u_j(s) ds = \int_0^t  u_1'(s) u_j'(s) + g_{1j}(s)ds. 
\end{equation}
Then, the lemma follows from Corollary \ref{cor_u1un}.
\end{proof}

\begin{lemma}\label{lemma13}
Under the hypothesis of Proposition \ref{prop:case12},
\begin{equation*}
 \int_0^{t}   u_1'(s) u_j'(s) + g_{1j}(s)ds
\end{equation*} 
converges as $t \to t_B^-$. 
\end{lemma}

\begin{proof}
If \eqref{t_case1} holds in Proposition \ref{prop:case12}, then the lemma immediately follows from \eqref{g1n_1} and \eqref{u1un0}. 

\sk In the case of \eqref{t_case2}, from \eqref{general_u2} and \eqref{general_u1} we have that
\begin{align*}
u_1'(s) &= \int_{t_B}^s g_1(\tau) \cos \omega (s-\tau) d\tau,\\
u_j'(s) &= \lambda_{j0} \cos \omega s - \omega \sin \omega s + \int_0^s g_j(\tau) \cos \omega(s-\tau) d\tau,
\end{align*}
and 
\begin{align*}
\int_0^{t} u_1'(s)u_j'(s) ds & = \int_0^{t} \left[\int_{t_B}^s g_1(x) \cos \omega (s-x) d x \left(\lambda_{j0} \cos \omega s - \omega \sin \omega s\right) \right]ds \\ 
&+ \int_0^t \left[\int_{t_B}^s g_1(x) \cos \omega (s-x) dx \int_0^s g_j(y) \cos \omega (s-y) dy \right]ds \\
& =: \Rom{1} + \Rom{2}.
\end{align*}
We notice that  $\frac{d}{dt}\Rom{1} \to 0$ as $t\to t_B^-$, because $g_1 \in L^1(0,t_B)$. It follows that $\Rom{1}$ also converges. Thus, it suffices to show that 
\begin{equation*}
h(t) := \Rom{2} + \int_0^t g_{1j}(s) ds 
\end{equation*}
converges.

\sk Changing the order of integration yields
\begin{equation*}
\Rom{2} = \int_0^t \int_{t_B}^y g_1(x) g_j(y)  \left( \frac1{2\omega} \sin\omega(x-y) + \frac{x-y}{2} \cos \omega(x-y)\right) dx dy
\end{equation*}
and the integral representation of $u_1$, \eqref{general_u2}, yields 
\begin{align*}
\int_0^t g_{1j}(y) dy &=  \int_0^t g_j(y) u_1(y) dy \\
& = \int_0^t \int_{t_B}^y g_j(y) g_1(x) \frac{1}{\omega} \sin \omega (y-x) dx   dy.
\end{align*}
We combine the two equations to obtain
\begin{align*}
h(t) = \frac{1}{2\omega} \int_0^t \int_{t_B}^y g_1(x) g_j(y) \left[ \omega(x-y) \cos\omega(x-y) - \sin\omega(x-y)\right] dx dy. 
\end{align*}  
Now, take $0< t_0 < t_B$ such that
\begin{equation*}
\omega (t_B -t_0) < \frac{\pi}{2}.
\end{equation*}
Then, for $t_0 \leq t < t_B$,
\begin{equation*}
h(t)  =  \frac{1}{2\omega} \int_{t_0}^t \int_{t_B}^y g_1(x) g_j(y) \left[ \omega(x-y) \cos\omega(x-y) - \sin\omega(x-y)\right] dx dy + h(t_0)
\end{equation*}
is a decreasing function, as the integrand $h'(t)$ is negative over the domain $(t_0, t_B)$. Furthermore, we observe from Lemma \ref{lemma12} and the convergence of $\Rom{1}$ that 
\begin{equation*}
h(t) = \left(\int_0^{t} u_1'(s)u_j'(s) ds +  \int_0^t g_{1j}(s) ds \right) - \Rom{1} 
\end{equation*}
is bounded. It follows that $h(t)$ converges as $t\to t_B^-$, as desired. 
\end{proof}

\sk 
\sk We proceed to show the convergence of $u_1'u_j$ and $u_1 u_j'$, which gives \eqref{keylemma} in Theorem \ref{mainthm}.
\begin{thm}\label{prop15}
Suppose that the maximum interval of existence for \eqref{main} is $[0,t_B)$ for some $0<t_B<\infty$. 
Define $u_j$ as \eqref{def_ui}. If $ J < j \leq n$, then there exist $0 \leq q_j \leq p_j$ such that
\begin{equation*}
\lim_{t\to t_B^-} u'_1(t) u_j(t) = -p_j, \quad \lim_{t\to t_B^-} u_1(t) u_j'(t) = q_j  
\end{equation*}
\end{thm}
\begin{proof}
The convergence of $u_1'u_j$ follows from Lemma \ref{lemma13} together with \eqref{g1n_1}, and the convergence of $u_1u_j'$ follows from \eqref{abel_Thm}.  

\sk Clearly, $p_j,q_j \geq 0$ and $p_j+q_j = -(\lambda_{1,0}-\lambda_{j,0})$, by \eqref{abel_Thm}. Furthermore, one can show  that $0 \leq q_j \leq p_j$. Suppose that $p_j < q_j$. Then, there exists $t_1 \in (0,t_B)$ such that  if $t_1 < t < t_B$, then
\begin{align*}
\frac{\lambda_j(t)}{-\lambda_1(t)} > 1 
\end{align*}
and 
\begin{equation*}
\lambda_j'(t) = -\lambda_j^2(t) + \frac{k}{n}(\rho(t) -c_b) < -\lambda_1^2(t)  + \frac{k}{n}(\rho(t) -c_b) = \lambda_1'(t).
\end{equation*}
Integration over $[t_1,t]$ yields
\begin{equation*}
\lambda_j(t) - \lambda_j(t_1) < \lambda_1(t) - \lambda_1(t_1)
\end{equation*}
which contradicts the fact that $\lambda_1 \to -\infty$ and $\lambda_j \to +\infty$.
\end{proof}  
\sk From now on, we let $p$ and $q$ denote $p_n$ and $q_n$, respectively. Then from Theorem \ref{prop15} either 
\begin{equation*}
p>q
\end{equation*}
or
\begin{equation*}
p=q
\end{equation*}
must hold. We investigate the solution behaviors stated in Theorem \ref{mainthm+} by considering these cases in the following two sections. Indeed, we obtain (c) of (\rom{2}) in Theorem \ref{mainthm+} by assuming that $p=q$, and all the other cases follow from $p > q$.

\section{The case $\boldsymbol{p > q}$}
\sk In this section, we describe the behaviors of blow-up solutions of \eqref{main} assuming that 
\begin{equation*}
p > q.
\end{equation*}

\sk We first state a technical lemma.  
\begin{lemma}\label{lemma:R}
Suppose that a function $R(t)$ defined in $[0,t_B)$ satisfies 
\begin{equation*}
(t_B -t) R(t) \to 0 \quad \text{~as~} t \to t_B^-.
\end{equation*}
Then, 
\begin{align*}
&\lim_{t\to t_B^-} (t_B-t)\int_0^t R^2(s) d s =0,\\
&\lim_{t\to t_B^-} (t_B-t) \int_0^t\frac{1}{t_B-s} R(s) ds =0.
\end{align*}
Furthermore, for any $0<\ep<1$ there exists $M>0$ such that 
\begin{align}\label{lemma:R_3}
\frac{(t_B-t)^\ep}{M} < e^{-\int_0^t R(s) d s} < \frac{M}{(t_B-t)^\ep}.
\end{align} 
\end{lemma}  
\begin{proof}
The first two limits follow from l'H\^{o}pital's rule. 
Let $0 < \ep <1$. Then, because $\lim_{s\to t_B^-}(t_B-s) R(s) = 0$, there exists $t_1 \in (0,t_B)$ such that for all $t_1 < s < t_B $, 
\begin{equation*}
(t_B -s) | R(s)| < \ep.
\end{equation*}
Then,  for $t > t_1$,
\begin{align*}
\left| \int_0^t R(s) d s \right| &\leq \int_0^t (t_B-s) |R(s)| \frac{1}{t_B -s} ds\\
&< \ep \int_{t_1}^t \frac{1}{t_B -s} ds + \int_0^{t_1} |R(s)| ds \\
& \leq -\ep \ln(t_B - t) + C, 
\end{align*}
for some constant $C$ that is independent of $t$. With $M=e^C$ it follows that 
\begin{align*}
 \frac{(t_B-t)^\ep}{M} < e^{-\int_0^t R(s) d s} < \frac{M}{(t_B-t)^\ep}.
\end{align*}
\end{proof}

\sk Because of  \eqref{u1un0} and Theorem \ref{prop15}, we set $(u_1u_n)(t_B)=0$ and $(u_1u_n)'(t_B)=-p+q <0$. Then, for some $\eta(t)$ such that
\begin{equation}\label{condition_eta}
\eta(t_B)=0,\quad \eta'(t_B)=0,
\end{equation}
it holds that 
\begin{equation*}
(u_1u_n)(t) = (p-q)(t_B-t) + \eta(t).
\end{equation*}  
It follows that 
\begin{align*}
\lambda_1(t) - \lambda_n(t) & = \frac{\lambda_{1,0}-\lambda_{n,0}}{(u_1u_n)(t)}  = \frac{\lambda_{1,0}-\lambda_{n,0}}{(p-q)(t_B-t) + \eta(t)} = \frac{-p -q}{(p-q)(t_B-t) + \eta(t)},\\
\lambda_1(t) + \lambda_n(t) &= \frac{(u_1u_n)'(t)}{u_1u_n(t)}  =\frac{-(p-q)+ \eta'(t)}{(p-q)(t_B-t) + \eta(t)}.
\end{align*}
Hence, 
\begin{align*}
\lambda_1(t) &= \frac{-p + \eta'(t)/2}{(p-q)(t_B-t) + \eta(t)},\\
\lambda_n(t) &= \frac{q + \eta'(t)/2}{(p-q)(t_B-t) + \eta(t)}.
\end{align*}
Owing to \eqref{condition_eta} we have the following forms: 
\begin{align*}
\lambda_1(t) &= \frac{-p}{p-q} \frac{1}{t_B-t} + R_1(t),\\
\lambda_n(t) &= \frac{q}{p-q} \frac{1}{t_B-t} + R_n(t),
\end{align*}
where $R_j(t)$ ($j=1,n$) satisfies $\lim_{t\to t_B^-}R_j(t) (t_B -t) =0$.

\sk Let 
  \begin{equation*}
   \lambda_j (t) = \frac{\xi_j}{t_B-t} + R_j \quad (j=1,n),
  \end{equation*}
with 
\begin{equation}\label{def_xi}
\xi_1:= \frac{-p}{p-q}, \quad \xi_n := \frac{q}{p-q}.
\end{equation}
Substituting this  into the main equation \eqref{lambda} yields 
  \begin{equation}\label{eqn_R_j}
  R_j'(t)= -\frac{\xi_j^2 +\xi_j}{(t_B-t)^2} - R_j ^2(t) - \frac{2\xi_j}{t_B-t}R_j(t) + \frac{k\rho_0}{n} e^{-\int_0^t \lambda(s) ds} - \omega^2.
  \end{equation}
Integrating  over $(0,t)$ and multiplying by $(t_B-t)$ give 
  \begin{align*}
  (t_B-t)R_j(t) &= -(\xi_j^2 +\xi_j) -  (t_B-t)\int_{0}^t \left[ R_j^2(\tau) + \frac{2\xi_j}{t_B-\tau} R_j(\tau)\right] d\tau \\
   &~ + \frac{k\rho_0}{n} (t_B-t)\int_{0}^t e^{-\int_0^\tau \lambda(s) ds}  d\tau  \\
   &~ + (t_B-t)\left[ \frac{\xi_j^2 +\xi_j}{t_B }-  \omega^2 t + R_j(0)\right]. 
      \end{align*}
Because $(t_B-t)\int_0^t \left[ R_j^2(\tau) + {2\xi_j R_j(\tau)}/(t_B-\tau)  \right] d\tau $ converges to $0$ as $t\to t_B^-$ by Lemma \ref{lemma:R}, we obtain the following quadratic equation for $\xi$:
  \begin{equation}\label{quadratic}
  \xi^2+\xi - \frac{k\rho_0}{n} \lim_{t\to t_B^-} (t_B-t)\int_{0}^t e^{-\int_0^\tau \lambda(s) ds} d\tau =0.
  \end{equation}
Here, $\xi = \xi_1, \xi_n$, for which the limit in \eqref{quadratic} must exist.

\sk Owing to Lemma \ref{lem_3} together with Theorem \ref{prop10}, we have that
\begin{align}
\lambda_i (t) &= \frac{\xi_1}{t_B-t} + R_i(t), \quad R_i(t) = R_1(t), \quad 1 \leq i \leq J, \label{def_lambda1}\\
\lambda_i (t) &= \frac{\xi_n}{t_B-t} + R_i(t), \quad J < i \leq n,\label{def_lambda2}
\end{align} 
where $  \lim_{t\to t_B^-} (t_B-t) R_i(t) =0$ for all $1\leq i \leq n$. It follows that
\begin{align}
\lambda(t) = \frac{-pJ + q(n-J)}{p-q} \frac{1}{t_B-t} + \sum_{i=1}^n R_i(t) = \frac\gamma{t_B-t} + R(t),\label{def_lambda}
\end{align}
where
\begin{align}
\gamma &:= \dis \frac{-pJ + q(n-J)}{p-q},\label{def_gamma}\\
R(t) &:= \sum_{i=1}^n R_i(t),\quad \lim_{t\to t_B^-} (t_B-t)R(t) =0.\nonumber
\end{align}
Now, we evaluate the limit in \eqref{quadratic} as follows. Note that
\begin{equation}\label{temp_21}
\int_{0}^t e^{-\int_0^\tau \lambda(s) ds} d\tau 
=  t_B^{-\gamma} \int_0^t (t_B-\tau)^\gamma  e^{-\int_0^\tau R(s) ds} d\tau.
\end{equation}
If follows from \eqref{lemma:R_3} that for any $ 0 < \ep < 1$, there exists $M>0$ such that 
\begin{equation*}
\frac{ t_B-t}{M}  \int_0^t (t_B- \tau)^{\gamma + \ep} d \tau  < (t_B-t) \int_0^t (t_B-\tau)^\gamma  e^{-\int_0^\tau R(s) ds}  d\tau < M (t_B-t) \int_0^t (t_B -\tau)^{\gamma - \ep} d \tau. 
\end{equation*}
Assume that $\gamma +2 <0$. Then, the lower bound     
\begin{equation*}
\frac{-1}{M(\gamma+1+\ep)} \left[ (t_B-t)^{\gamma+2 +\ep} - t_B^{\gamma +1 +\ep} (t_B-t) \right]  \to + \infty \text{~as~} t\to t_B^-
\end{equation*}
by taking $\ep$ sufficiently small so that $\gamma + 2 + \ep <0$. This is not the case, as the limit in \eqref{quadratic} must converge, as previously mentioned. On the other hand,  $\gamma + 2 > 0$ implies that
the upper bound 
\begin{equation*}
\frac{-M}{\gamma+1 -\ep} \left[ (t_B-t)^{\gamma+2 -\ep} - t_B^{\gamma +1 -\ep} (t_B-t) \right]  \to 0 \text{~as~} t\to t_B^-
\end{equation*} 
by taking $\ep$ such that $\gamma+2 -\ep >0$ and  $\gamma - \ep  \neq -1$. This ensures that 
\begin{equation*}
\xi^2 + \xi =0. 
\end{equation*}
It follows that 
\begin{equation*}
\xi_1 = \frac{-p}{p-q}=-1, \quad \xi_n =\frac{q}{p-q}=0.
\end{equation*}
Substituting $q=0$ into \eqref{def_gamma} together with $\gamma + 2 >0$ then yields 
\begin{equation*}
J=1.
\end{equation*}

\sk Now, consider the case that $\gamma + 2 =0$. We first claim that 
\begin{subequations}\label{converge_R_2}
\begin{eqnarray}
&~&\lim_{t\to t_B^-} (t_B-t) \int_0^t (t_B-\tau)^{-2} e^{-\int_0^\tau R(s) ds} d\tau \label{temp10}\\
&=& \lim_{t\to t_B^-} e^{-\int_0^t R(s) ds}.\label{temp11}
\end{eqnarray}
\end{subequations} 
\sk We remark that, in general, the convergence of \eqref{temp10}, which we have already verified, does not guarantee the convergence of \eqref{temp11}, because \eqref{temp10} may converge for an oscillating divergent $\int_0^t R(s) ds$. However, the decay property of $R$ can eliminate this case. By integration by parts, 
\begin{align}
& ~ \lim_{t\to t_B^-} (t_B-t) \int_0^t (t_B-\tau)^{-2} e^{-\int_0^\tau R(s) ds} d\tau \nonumber\\
& = \lim_{t\to t_B^-} \left[ e^{-\int_0^t R(s) ds}  -\frac{t_B-t}{t_B} + (t_B-t) \int_0^t (t_B-\tau)^{-1} e^{-\int_0^\tau R(s) ds} R(\tau) d\tau  \right]. \label{converge_R_1}
\end{align} 
Recall that $(t_B-t)R(t) \to 0$ as $t\to t_B^-$. Then, there exists $t_1 \in (0,t_B)$ such that 
\begin{equation*}
(t_B -t)|R(t)| < 1, \quad t_1< t< t_B,
\end{equation*}
and  
\begin{align} 
&~~ \left|(t_B-t) \int_0^t (t_B-\tau)^{-1} e^{-\int_0^\tau R(s) ds} R(\tau) d\tau\right| \label{temp_3}\\
&\leq (t_B-t) \int_0^{t_1} (t_B-\tau)^{-1} e^{-\int_0^\tau R(s) ds} |R(\tau)| d\tau + (t_B-t) \int_{t_1}^t (t_B-\tau)^{-2} e^{-\int_0^\tau R(s) ds}  d\tau.\label{temp_2}
\end{align}
Because \eqref{temp10} converges, the second term in \eqref{temp_2} converges, and thus \eqref{temp_3} converges as $t\to t_B^-$. The convergence of $\exp \Big({-\int_0^{t_B} R(s) ds}\Big)$ follows from \eqref{converge_R_1}.  Now, apply l'H\^{o}pital's rule to obtain \eqref{converge_R_2}. 

\sk Thus, the case with $\gamma+2=0$ may be considered as either     
\begin{equation}\label{gamma=2_case1}
\lim_{t\to t_B^-} (t_B-t) \int_0^t (t_B-\tau)^\gamma  e^{-\int_0^\tau R(s) ds} d\tau =  \lim_{t\to t_B^-}  e^{-\int_0^t R(s) ds} = 0,
\end{equation}
or  
\begin{equation}\label{gamma=2_case2}
\lim_{t\to t_B^-} (t_B-t) \int_0^t (t_B-\tau)^\gamma  e^{-\int_0^\tau R(s) ds} d\tau = \lim_{t\to t_B^-}  e^{-\int_0^t R(s) ds} = R_0 > 0.
\end{equation} 
For the case that \eqref{gamma=2_case1}, a similar argument as that in the case for $\gamma +2 <0$ yields 
\begin{equation*}
\xi_1 = \frac{-p}{p-q}=-1, \quad \xi_n =\frac{q}{p-q}=0,
\end{equation*}
and  
\begin{equation*}
J=2.
\end{equation*} 
Furthermore, \eqref{gamma=2_case1} implies that 
\begin{equation*}
\lim_{t\to t_B^-}  \int_0^t R(s) ds  = \infty. 
\end{equation*} 
\sk For the case that \eqref{gamma=2_case2}, we deduce from \eqref{quadratic} and \eqref{temp_21} that 
\begin{equation*}
\xi^2 + \xi -\dis \frac{k \rho_0t_B^2 R_0}{n }=0,
\end{equation*}
and from \eqref{def_gamma} that 
\begin{equation*}
p(J-2) = q(n-J-2).
\end{equation*}
We divide this into two cases, by taking into account $p>q$:
\begin{align*}
&J=2,~ n=4 \quad \text{or}\\
&J\geq 3,~ n > 2J 
\end{align*}

\sk In summary, we have the following: 
\begin{thm}\label{thm:case1}
Suppose that $[0,t_B)$ be the maximum interval of existence for \eqref{main}. Define $u_i$ as \eqref{def_ui}, and let 
\begin{equation*}
\lim_{t\to t_B^-} u_1'(t) u_n(t) = -p, \quad \lim_{t\to t_B^-} u_1(t) u_n'(t) = q.
\end{equation*}
If $p>q$, then $\lambda_i$ ($i=1,2, \cdots,n$) and $\lambda$ can be represented by \eqref{def_lambda1}, \eqref{def_lambda2}, and \eqref{def_lambda}. Moreover, one of the following must hold, where $\xi = \xi_1,~ \xi_n$:
\begin{itemize}
\item[] \textbf{(1)}   $J=1$ and
\begin{equation*}
\xi^2 + \xi =0.
\end{equation*}
\item[] \textbf{(2-a)}  $J=2$, $ \lim_{t\to t_B^-}  \int_0^t  R(s) ds = \infty
$, and 
\begin{equation*}
\xi^2 + \xi =0.
\end{equation*}
\item[] \textbf{(2-b)} $J=2, ~ n=4$, $ \lim_{t\to t_B^-}   \exp({-\int_0^t R(s) ds}) = R_0 > 0 $, and 
\begin{equation*}
\xi^2 + \xi  - \frac{k \rho_0t_B^2 R_0}{4}=0.
\end{equation*}
\item[] \textbf{(3)} $J \geq 3, ~ n > 2J$, $  \lim_{t\to t_B^-}   \exp({-\int_0^t R(s) ds}) = R_0 > 0 $, and 
\begin{equation*}
\xi^2 + \xi  - \frac{k \rho_0t_B^2 R_0}{n}=0.
\end{equation*}
\end{itemize}
Furthermore, these cases imply (\rom{1}), (a), (b) of (\rom{2}), and (\rom{3}) in Theorem \ref{mainthm+}, respectively.
\end{thm}
  
\sk The remainder of the proof of Theorem \ref{thm:case1} is demonstrating the relations between the cases in Theorem \ref{thm:case1} and in Theorem \ref{mainthm+}.

\sk Assuming \textbf{(1)}, we immediately have the following representation of $\lambda_i$: 
\begin{eqnarray*}
\lambda_i (t) &=& 
\begin{cases}
\dis\frac{-1}{t_B-t} + R_1(t), \quad &i=1,\\
 R_i(t),\quad  &2 \leq i \leq  n,
 \end{cases}\\
\lambda (t) &=& \frac{-1}{t_B-t} + R(t), \quad R(t) = \sum_{i=1}^n R_i(t).
\end{eqnarray*}
Although $\dis \lim_{t\to t_B^-} (t_B -t) R_i(t) =0$ for $i=1,2,\cdots, n$, we require the integrability of $R_i$ to obtain (\rom{1}) in Theorem \ref{mainthm+}. Indeed, this is the case.   
\begin{lemma}\label{lem17}
Assuming \textbf{(1)} in Theorem \ref{thm:case1},  
\begin{equation*}
\lambda_i  \in L^1(0,t_B), \quad i=2,\cdots,n,
\end{equation*}
and 
\begin{equation*}
 \int_0^{t_B} R_1(s)ds = C.
\end{equation*}
\end{lemma}
\begin{proof}
Let $i=2,3,\cdots, n$. Then, we deduce that 
\begin{align}
\lambda_i(t) &= -\int_{0}^t \lambda^2_i(s)ds +\frac{k\rho_0 t_B}{n} \int_{0}^t \frac{1}{t_B-s} e^{-\int_0^s R(\tau) d\tau}ds  -\frac{k c_b}{n}t + \lambda_{i,0}\label{claim1-2}\\
&\leq  \int_{0}^t \frac{1}{t_B-s} e^{-\int_0^s R(\tau) d\tau} ds  + \lambda_{i,0}.\label{claim1-21}
\end{align}
Multiplying by $(t_B-t)^{1/2}$ yields  
\begin{equation}\label{4.61} 
(t_B-t)^{1/2} \lambda_i(t) \leq (t_B-t)^{1/2} \int_{0}^t \frac{1}{t_B-s} e^{-\int_0^s R(\tau) d\tau} ds  + (t_B-t)^{1/2}\lambda_{i,0}.
\end{equation}
Now, we take $\ep = 1/3$ in \eqref{lemma:R_3} to obtain  
\begin{equation*}
e^{-\int_0^s R(\tau) d\tau} \leq \frac{M}{(t_B-s)^{1/3}}.
\end{equation*} 
Then, we observe that the right-hand side of \eqref{4.61} converges to $0$. Thus, 
 \begin{equation*}
\lim_{t\to t_B^-} (t_B-t)^{1/2} \lambda_i(t) =0, 
 \end{equation*}
because $\lambda_i > 0 $ near $t_B$. This implies that
\begin{equation}\label{temp41}
\lambda_i  \in L^1(0,t_B), \quad i=2,3,\cdots,n.
\end{equation}
To demonstrate the convergence of $\int_0^{t} R_1(s)ds$, 
we deduce from \eqref{differenc_lambda} that 
\begin{align}\label{4.11}
(t_B -t) (\lambda_1(t) - \lambda_n(t)) &= (\lambda_{1,0} -\lambda_{n,0}) t_B e^{-\int_0^t R_1(s) + \lambda_n(s) ds }. 
\end{align}
Because the left-hand side converges to $-1$ assuming \textbf{(1)}, there exists a constant $C_1$ such that   
\begin{equation} \label{r_1+lambda_n}
\int_0^{t_B} R_1(s) + \lambda_n(s) ds  = C_1,
\end{equation}
and thus \eqref{temp41}, $\lambda_n\in L^1(0,t_B)$, yields 
\begin{align*}
\int_0^{t_B} R_1(s) ds   = C. 
\end{align*} 
\end{proof}
\sk Lemma \ref{lem17} enhances the estimate \eqref{lemma:R_3} as
\begin{equation*}
0< \lim_{t\to t_B^-} e^{-\int_0^{t} R(s) ds} < \infty.
\end{equation*}
Immediately, we obtain   
\begin{equation*}
\rho(t) = \O \big( \frac{1}{t_B -t}\big) \quad \text{as~} t \to t_B^-.
\end{equation*}
Furthermore, it follows from \eqref{claim1-21} that $\lambda_i$ is at most $\O(\ln (t_B -t))$ for $i=2,3,\cdots, n$. Then, $\lambda_i \in L^2(0,t_B)$, and applying \eqref{claim1-2} again yields 
\begin{equation*}
\lambda_i (t) = \O (\ln(t_B -t)) \quad i=2,3,\cdots,n.
\end{equation*} 
This shows that \textbf{(1)} implies (\rom{1}) in Theorem \ref{mainthm+}. 

\sk In the case of \textbf{(2-a)} in Theorem \ref{thm:case1},  
\begin{eqnarray*}
\lambda_i(t) &=& 
\begin{cases}
\dis\frac{-1}{t_B-t} + R_i(t), \quad &i=1, 2,  \\
 R_i(t),\quad  & 3 \leq i \leq n, 
 \end{cases}\\
  \lambda (t) &=& \frac{-2}{t_B-t} + R(t), \quad R(t) = \sum_{i=1}^n R_i(t), \quad R_1(t) = R_2(t).
\end{eqnarray*}
Now, let $3\leq i \leq n$. Then, similar to the derivation of \eqref{r_1+lambda_n}, we have that 
\begin{equation}\label{4.7}
 \int_0^{t_B} R_1(s) + \lambda_i(s) ds  = C_i.
\end{equation}
If $\int_0^{t} \lambda_i(s) ds $ is assumed to converge, then 
$\int_0^t R_1(s) ds$, and thus $\int_0^{t} R(s) ds$ converges, which does not belong to \textbf{(2-a)}. Taking into account $\lambda_i \to \infty$, 
we must have 
\begin{align*}
\lim_{t\to t_B^-} \int_0^t \lambda_i (s) ds = \infty, \quad i=3,4, \cdots, n.
\end{align*}
Then, \eqref{4.7} yields 
\begin{equation}\label{4.8}
\lim_{t\to t_B^-} \int_0^t R_1(s)  ds = - \infty.\\
\end{equation}
Summing \eqref{4.7} over $i =3, 4, \cdots, n$ yields that for some constant $C$,  
\begin{equation}\label{R_1+lam_j}
 \int_0^{t_B} R(s) + (n-4)R_1(s) ds  = C.
\end{equation}
Because $ \lim_{t\to t_B^-} \int_0^{t} R(s) ds = \infty $ in \textbf{(2-a)},  we have that 
\begin{equation*}
n \geq 5,
\end{equation*}
and it follows that  
\begin{equation*}
\rho(t) = o \left(\frac{1}{(t_B -t)^2}\right) \text{~as~} t\to t_B^-.
\end{equation*}
Hence, we conclude that \textbf{(2-a)} in Theorem \ref{thm:case1} implies (a) of (\rom{2}) in Theorem \ref{mainthm+}.

\sk Now, we consider the case of \textbf{(2-b)}. Because the solutions to the characteristic equation \eqref{quadratic} are $\xi_1= -p/(p-q)$ and $\xi_4=q/({p-q})$, it follows that 
\begin{equation}\label{4.12}
\frac{pq}{(p-q)^2} = \frac{k\rho_0 t^2_B R_0}{4}
\end{equation}
and 
\begin{eqnarray*}
\lambda_i(t) &=& 
\begin{cases}
\dis\frac{\xi_1}{t_B-t} + R_i(t), \quad & i = 1, 2, \\
\dis\frac{\xi_4}{t_B-t} + R_i(t), \quad  & i =3, 4, 
 \end{cases}\\
\lambda (t) &=& \frac{-2}{t_B-t} + R(t), \quad R(t) = \sum_{i=1}^4 R_i(t), \quad R_1(t) = R_2(t). 
\end{eqnarray*}
Note that the representation of $\lambda$ follows from $\xi_1+ \xi_4 = -1$, and the representation of $\lambda_3$ (i.e., $\xi_3 = \xi_4$)  follows from Lemma \ref{lem_3}. Because $\lim_{t\to t_B^-}   \exp({-\int_0^t R(s) ds}) = R_0 > 0$, we immediately we obtain that
\begin{equation*}
\rho(t) = \O \Big(\frac{1}{(t_B-t)^2}\Big)\quad \text{as~} t\to t_B^-. 
\end{equation*}
Similar to \eqref{4.11}, we deduce that 
\begin{align*}
(t_B-t)(\lambda_1(t) - \lambda_i(t)) = (\lambda_{1,0}-\lambda_{i,0}) t_B e^{-\int_0^t R_1(s) + R_i(s) ds},\quad i=3,4.
\end{align*}    
Sending $t\to t_B^-$ and multiplying the two equations for $i=3,4$ yield that 
\begin{equation}\label{4.13}
\frac{(p+q)^2}{(p-q)^2} = A_0 t_B^2 R_0.  
\end{equation} 
Recall that
\begin{equation*}
A_0:= (\lambda_{1,0}-\lambda_{3,0})(\lambda_{1,0}-\lambda_{4,0}).
\end{equation*}
Then, we combine \eqref{4.12} and \eqref{4.13} to obtain 
\begin{equation*}
(p-q)^2 = 4\Big(\frac{A_0}{k \rho_0}-1\Big)pq.
\end{equation*}
Thus, it must hold that 
\begin{equation}\label{case_b-2}
{A_0}>{k \rho_0}. 
\end{equation}
Furthermore, we obtain representations of $\xi_1$ and $\xi_4$ in terms of the given parameters. Indeed, we have 
\begin{align*}
\xi_1 &= -\frac{1}{2} -\frac{1}{2}\sqrt{\frac{A_0}{A_0 - k \rho_0}}, \\
\xi_3 &=\xi_4= -\frac{1}{2} + \frac{1}{2}\sqrt{\frac{A_0}{A_0 - k \rho_0}},
\end{align*} 
as described in (b) of (\rom{2}) in Theorem \ref{mainthm+}. 

\sk In the case of \textbf{(3)} in Theorem \ref{thm:case1}, we have that
\begin{align*}
\lim_{t\to t_B^-} (t_B-t)\lambda_1 (t) &= \frac{-p}{p-q},\\
\lim_{t\to t_B^-} (t_B-t)\lambda_i (t) &= \frac{q}{p-q}, \quad J+1 \leq i \leq n.
\end{align*}
The behavior of $\rho$, 
\begin{align*}
\rho (t) = \O\Big(\frac{1}{(t_B-t)^{2}}\Big) \quad \text{as~} t\to t_B^-, 
\end{align*}
follows from  $(-pJ + q(n-J))/(p-q) = -2$ and 
$ \exp ({-\int_0^{t_B} R(s) ds }) = R_0 $. 
This shows that \textbf{(3)} implies (\rom{3}) in Theorem \ref{mainthm+}.

\section{The case $\boldsymbol{p=q}$}
\sk In this section, we investigate the blow-up solution behaviors when
\begin{equation}\label{p=q}
p=q \left( = \frac{\lambda_{n,0}- \lambda_{1,0}}{2}\right).
\end{equation}
As previously noted, understanding the behaviors of $\lambda_i'$ near $t_B$ is essential. One technique to achieve this is to compare the behaviors of $\lambda_i^2$ and $\rho$ from \eqref{lambda}. However, the main difficulty lies in the fact that the condition \eqref{p=q} implies that the leading singular terms of $\int \lambda_i^2$ and $k/n \int \rho$ are the same. Indeed, integrating \eqref{lambda} yields
\begin{align*}
\lambda_1(t) - \lambda_{1,0} = - \int_0^t \lambda_1^2(s) ds + \frac{k}{n} \int_0^t \rho(s) -\omega^2 ds \to - \infty,\\
 \lambda_n(t) - \lambda_{n,0} = - \int_0^t \lambda_n^2(s) ds + \frac{k}{n} \int_0^t \rho(s) -\omega^2 ds \to +\infty,
\end{align*}
implying  that in a neighborhood of $t_B$,
\begin{equation}\label{temp51}
\int_0^t \lambda^2_n(s) ds < \int_0^t \rho(s) ds < \int_0^t \lambda^2_1(s) ds. 
\end{equation}
However, the condition \eqref{p=q} yields  
\begin{equation}\label{5.8}
\lim_{t\to t_B^-} \frac{\lambda_1(t)}{\lambda_n(t)} = \lim_{t\to t_B^-}\frac{u'_1(t) u_n(t)}{u_1(t) u_n'(t)} = -1,
\end{equation}
which indicates that the leading singular terms of all integrals in \eqref{temp51} are the same. For this reason, we study the case of \eqref{p=q} by examining the second singular terms of $\int \lambda_i^2$ and $\int \rho$. We remark that one cannot compare $\lambda_i^2$ and $k \rho/n$ directly as Proposition \ref{lem:rho} demonstrates the behavior of $\int \rho$ rather than $\rho$. Furthermore, we notice that the case \eqref{p=q} occurs only in the case of \eqref{t_case2} in Proposition \ref{prop:case12}. Indeed, \eqref{t_case1} implies that 
\begin{equation*}
\lim_{t\to t_B}\int_{0}^t \lambda_n(s) ds = \lim_{t\to t_B}\int_{0}^t \frac{u_n'(s)}{u_n(s)} ds = \ln \beta_n  < \infty.  
\end{equation*}
Assuming \eqref{p=q}, we have observed \eqref{5.8}, which implies that $\lambda_1 \in L^1(0,t_B)$. Thus, $\lambda_i\in L^1(0,t_B)$ for all $i$, and thus $\rho$ is bounded. This contracts Proposition \ref{lem:rho}. More precisely,  $\lambda_i \in L^1(0,t_B) ~(i>J)$, which is a necessary and sufficient condition for the convergence of $u_i ~(i>J)$ or \eqref{t_case1} in Proposition \ref{prop:case12}, only holds in (\rom{1}) in Theorem \ref{mainthm+}. That is, (\rom{1}) is equivalent to \eqref{t_case1}, and all other cases in Theorem \ref{mainthm+} are associated with \eqref{t_case2} in Proposition \ref{prop:case12}.  

\sk We define $\eta$ as
\begin{equation}\label{claim2-9}
\frac{u_1'}{u_1} + \frac{u_n'}{u_n} = -2 \eta. 
\end{equation}
Because $\lim_{t\to t_B^-}(u_1u_n)(t) =0$ in Corollary \ref{cor_u1un} and $\lim_{t\to t_B^-}(u_1u_n)'(t)=0$ from the condition \eqref{p=q}, $\eta$ satisfies 
\begin{align}
\lim_{t\to t_B^-} \eta(t)(u_1u_n)(t) = -\lim_{t\to t_B^-} \frac{(u_1u_n)'(t)}{2} =0, \label{claim2-1}\\
\lim_{t\to t_B^-} \int_0^t \eta(s) ds = -\lim_{t\to t_B^-} \frac{\ln ((u_1u_n)(t))}{2} = \infty. \label{claim2-10}
\end{align}  
We remark that the behavior of $\eta$ near $t_B$ is not clear at this point, owing to the highly oscillating type \eqref{def2}.  
  
\sk Recall \eqref{abel_Thm} or that for all $t\in (0,t_B)$,
\begin{equation}\label{claim2-8}
\frac{u_1'}{u_1} - \frac{u_n'}{u_n} = -2 \frac{p}{u_1u_n}.
\end{equation}
Then, together with \eqref{claim2-9}, we have that  
\begin{align}
\lambda_1 &= \frac{u_1'}{u_1} = -\frac{p}{u_1u_n} - \eta,
\label{claim2-4}\\
\lambda_n &= \frac{u_n'}{u_n} = \frac{p}{u_1u_n} -\eta.
\label{claim2-5}
\end{align}
Substituting these representations into the main equation \eqref{lambda} yields
\begin{align}
\lambda_1' &= - \lambda_1^2 + \frac{k}{n}\rho -\omega^2 =   - \left[\left( \frac{p}{u_1u_n}\right)^2 + 2\frac{p\eta}{u_1u_n} + \eta^2 \right] + \frac{k}{n} \rho - \omega^2, \label{5.3}\\ 
\lambda_n' &= - \lambda_1^2 + \frac{k}{n}\rho -\omega^2 = - \left[\left( \frac{p}{u_1u_n}\right)^2 - 2\frac{p\eta}{u_1u_n} + \eta^2  \right]+ \frac{k}{n} \rho - \omega^2.\label{5.4}
\end{align}
Owing to the property of $\eta$ in \eqref{claim2-1}, we obtain
\begin{align*}
&\dis \lim_{t\to t_B^-}\frac{\dis\int_0^t \frac{p\eta(s)}{(u_1u_n)(s)} ds} {\dis\int_0^t \left( \frac{p}{(u_1u_n)(s)}\right)^2 ds}=0.
\end{align*}
Thus, the leading singular term of $\int_0^t \lambda_i^2(s) ds$ ($i=1,n$) is 
\begin{equation*}
\int_0^t \left( \frac{p}{(u_1u_n)(s)}\right)^2 ds,  
\end{equation*}
and this should be the same as the leading singular term of $k/n\int_0^t \rho(s) ds $,  otherwise the integrations of \eqref{5.3} and \eqref{5.4} yield that $\lambda_1\lambda_n >0$ near $t_B$. Now, we define $\delta$ as 
\begin{equation}\label{claim2-3}
 \int_0^t \left( \frac{p}{(u_1u_n)(s)}\right)^2 ds + \delta(t) := \int_0^t \frac{k}{n} \rho(s) -\omega^2 ds, 
\end{equation}
satisfying 
\begin{equation}\label{claim2-2}
\dis \lim_{t\to t_B^-} \frac{\delta(t)}{ \dis \int_0^t \left( \frac{p}{(u_1u_n)(s)}\right)^2 ds } = 0, \quad \delta(0) =0.
\end{equation}
It follows from \eqref{5.3} and \eqref{5.4} that
\begin{align}
\lambda_1(t) -\lambda_{1,0} &= \int_0^t \Big(-2\frac{p\eta(s)}{(u_1u_n)(s)} - \eta^2(s)   \Big) ds  + \delta(t), \label{claim2-6}\\
\lambda_n(t) -\lambda_{n,0} &= \int_0^t \Big(2\frac{p\eta(s)}{(u_1u_n)(s)} - \eta^2(s)  \Big) ds  + \delta(t). \label{claim2-7}
\end{align}

\sk Now, we present a technical lemma. In Corollary \ref{cor_u1un}, we showed that $u_1u_n \to 0$ as $t$ tends to $t_B$. Thus, one may expect that for some $\theta>1$, $u_1u_n^\theta $ converges to a nonzero constant with assuming \eqref{t_case2}. However, this is not the case,  at least when $p=q$.  
\begin{lemma}\label{lemma:claim2_2}
Assume the hypothesis of Theorem \ref{thm:case1}, and suppose that 
\begin{equation*}
p=q.
\end{equation*}
Then, for any $ \theta \leq 1$, 
\begin{equation*}
\lim_{t\to t_B^-} (u_1u_n^{\theta})(t) =0.
\end{equation*} 
For $\theta >1$, if the convergence of $u_1 u_n^{\theta}$ is assumed, then 
\begin{equation*}
\lim_{t\to t_B^-} (u_1u_n^{\theta})(t) =0.
\end{equation*}
\end{lemma}   
\begin{proof}
Recall that $p=q$ only occurs in \eqref{t_case2}, i.e., $u_n \to \infty$.  Then, it clearly holds that $\lim_{t\to t_B^-} (u_1u_n^{\theta})(t) =0$ for $\theta \leq 1$, because $\lim_{t\to t_B^-}(u_1u_n)(t) =0$ in Corollary \ref{cor_u1un}.
 
\sk Let $\theta >1$ and      
\begin{equation*}
\lim_{t\to t_B^-} (u_1u_n^{\theta})(t) = C.   
\end{equation*} 
We deduce from \eqref{abel_Thm} that 
\begin{align*}
\left(\frac{u_i(t)}{u_j(t)}\right)' &= \frac{\lambda_{i,0}-\lambda_{j,0}}{u_j^2(t)},\\
u_i(t) &= u_j(t) + u_j(t) (\lambda_{i,0}-\lambda_{j,0})\int_0^t \frac{1}{u_j^2(s)} ds.
\end{align*}
Multiplying by $u_n^\theta$ in the equation for $i=1, j =n$ yields  
\begin{equation*}
(u_1 u_n^\theta)(t) = u_n^{\theta+1}(t) \left(1+  (\lambda_{1,0}-\lambda_{n,0})\int_0^t \frac{1}{u_n^2(s)} ds\right).
\end{equation*}  
Then, $ 1+  (\lambda_{10}-\lambda_{n,0})\int_0^{t_B} 1/{u_n^2(s)} ds =0$, because the left-hand side converges to $C$, and $u_n^{\theta+1} \to \infty$.  Apply l'H\^{o}pital's rule to the right-hand side, to yield   
\begin{align*}
\dis\lim_{t\to t_B^-} (u_1u_n^\theta)(t) & = \lim_{t\to t_B^-} \frac{\lambda_{1,0}-\lambda_{n,0}}{-\theta -1} \frac{u_n^\theta(t)}{u_n'(t)} \\
& = \lim_{t\to t_B^-} \frac{\lambda_{1,0}-\lambda_{n,0}}{-\theta -1} \frac{(u_n^\theta u_1)(t)}{(u_n' u_1)(t)}.  
\end{align*}
Because the final limit exists, 
\begin{equation*}
C = \frac{2p}{\theta + 1} \frac{C}{p} = \frac{2C}{\theta+1}.
\end{equation*}
Then, as we assumed that $\theta >1 $, it follows that $C=0$, as desired. 
\end{proof}

\sk In the following theorem, we claim that the case of $p=q$ implies (c) of (\rom{2}) in Theorem \ref{mainthm+}.  
\begin{thm}\label{thm_p=q}
Suppose that $[0,t_B)$ be the maximum interval of existence for \eqref{main}. Define $u_i$ as \eqref{def_ui}, and let 
\begin{equation*}
\lim_{t\to t_B^-} u_1'(t) u_n(t) = -p, \quad \lim_{t\to t_B^-} u_1(t) u_n'(t) = q.
\end{equation*}
If   
\begin{equation*}
p=q,
\end{equation*}
then $J=2$, $n=4$, and  
\begin{equation*}
(\lambda_{1,0}-\lambda_{3,0})(\lambda_{1,0}-\lambda_{4,0})=:A_0 = k\rho_0.
\end{equation*}
Moreover, this implies (c) of (\rom{2}) in Theorem \ref{mainthm+}. 
\end{thm}

\proof
Recall \eqref{rho_u}
\begin{equation*}
\rho = \rho_0 \prod_{i=1}^n \frac{1}{u_i}. 
\end{equation*}
From \eqref{claim2-3} and \eqref{claim2-2},
\begin{equation*}
\dis \lim_{t\to t_B^-} \frac{\int_0^t (u_1u_n)^{-2}ds}{\int_0^t  \prod_{i=1}^n u_i^{-1} ds} = \frac{k\rho_0}{np^2}.
\end{equation*}
We apply Cauchy's mean value theorem, to obtain 
\begin{equation*}
\frac{\int_0^t (u_1u_n)^{-2} ds}{\int_0^t \prod_{i=1}^n u_i^{-1} ds} \frac{1- \frac{\int_0^{t_1} (u_1u_n)^{-2} ds}{\int_0^t (u_1u_n)^{-2} ds}}{1-\frac{\int_0^{t_1} \prod_{i=1}^n u_i^{-1} ds}{\int_0^{t} \prod_{i=1}^n u_i^{-1} ds} } = \frac{\prod_{i=1}^n u_i(\tau)} {u_1^2(\tau)u_n^2(\tau)}
\end{equation*}  
 for some $t_1<\tau<t$. Owing to the convergence of the left-hand side as $t\to t_B^-$, we can construct a sequence $\{\tau_l\}_{l=1}^\infty$ converging  to $t_B$ such that 
 \begin{equation*}
\lim_{l\to \infty} \frac{\prod_{i=1}^n u_i(\tau_l)} {u_1^2 (\tau_l)u_n^2(\tau_l)} = \frac{k \rho_0}{np^2}. 
 \end{equation*} 
It follows from Lemma \ref{lem_3}  that 
\begin{equation}\label{6.1}
\lim_{l\to \infty} u_1^{J-2}(\tau_l)u_n^{n-J-2}(\tau_l) = \frac{k \rho_0}{np^2} \prod_{j=J+1}^{n-1} \frac{\lambda_{n,0} -\lambda_{1,0}}{\lambda_{j,0}-\lambda_{1,0}}. 
\end{equation} 
If it is assumed that $J\neq 2$, then 
\begin{equation}\label{temp73}
\lim_{l\to \infty} u_1(\tau_l)u_n^{\frac{n-J-2}{J-2}}(\tau_l) =   \left(\frac{k \rho_0}{np^2} \prod_{j=J+1}^{n-1} \frac{\lambda_{n,0} -\lambda_{1,0}}{\lambda_{j,0}-\lambda_{1,0}} \right)^{\frac{1}{J-2}}. 
\end{equation}
If it is additionally assumed that $(n-J-2)/({J-2}) \leq 1$, then Lemma \ref{lemma:claim2_2} implies that 
\begin{equation*} 
\dis \left(\frac{k \rho_0}{np^2} \prod_{j=J+1}^{n-1} \frac{\lambda_{n,0} -\lambda_{1,0}}{\lambda_{j,0}-\lambda_{1,0}} \right)^{\frac{1}{J-2}} =0,
\end{equation*}
which is not possible. The assumption that $(n-J-2)/({J-2}) > 1$ also yields a contradiction. Indeed, under this assumption one can show that $u_1u_n^{({n-J-2})/({J-2})}$ is an increasing function in a neighborhood of $t_B$, by showing that  for any $\ep>0$ there exists $t_1 \in (0,t_B)$ such that      
\begin{equation}\label{temp91}
(u_1u_n^{1+\ep})' (t) >0,  \quad \quad t_1 < t < t_B  
\end{equation}
in the case of $p=q$. Then, together with  \eqref{temp73} we have  

\begin{equation*}
\lim_{t\to t_B^-} u_1(t)u_n^{\frac{n-J-2}{J-2}}(t) =  \left(\frac{k \rho_0}{np^2} \prod_{j=J+1}^{n-1} \frac{\lambda_{n0} -\lambda_{10}}{\lambda_{j0}-\lambda_{10}} \right)^{\frac{1}{J-2}}.   
\end{equation*} 
Now, we apply Lemma \ref{lemma:claim2_2} to obtain   
\begin{equation*}
\left(\frac{k \rho_0}{np^2} \prod_{j=J+1}^{n-1} \frac{\lambda_{n,0} -\lambda_{1,0}}{\lambda_{j,0}-\lambda_{1,0}}\right)^{\frac{1}{J-2}}  =0,
\end{equation*}
which is also not possible. 
 
\sk Hence, $J=2$. Because the case of $p=q$ is corresponds to \eqref{t_case2}, we must have that $n=4$. Moreover, substituting \eqref{p=q} into \eqref{6.1} with $J=2$ and $n=4$ yields
\begin{align}
k\rho_0 &= (\lambda_{4,0}-\lambda_{1,0})(\lambda_{3,0}-\lambda_{1,0}),\label{k_rho}
\end{align}
as desired. 

\sk It remains to verify the solution behaviors described in (c) of (\rom{2}) in Theorem \ref{mainthm+}. We first state a lemma describing the behavior of $\delta$ near $t_B$.    
\begin{lemma}\label{lemma_delta}
Under the hypothesis of Theorem \ref{thm_p=q}, $\delta$ defined in \eqref{claim2-3} satisfies  
\begin{equation*}
\lim_{t\to t_B^-} \delta'(t) = -\omega^2.
\end{equation*} 
\end{lemma}
\begin{proof}
We have shown that $J=2$ and $n=4$ when $p=q$. Thus, we deduce from 
 \eqref{claim2-3} and \eqref{rho_u} that
\begin{align*}
\delta'  = \frac{1}{u_1^2u_4} \left( \frac{k \rho_0}{4}\frac{1}{u_3} - \frac{p^2}{u_4} \right) -\omega^2.
\end{align*}
Using the representation in \eqref{temp61},
\begin{equation*}
u_3 = \frac{\lambda_{4,0}-\lambda_{3,0}}{\lambda_{4,0}-\lambda_{1,0}} u_1 + \frac{\lambda_{3,0}-\lambda_{1,0}}{\lambda_{4,0}-\lambda_{1,0}} u_4,
\end{equation*}
we have that 
\begin{equation}\label{ex_delta}
\delta'= \frac{-p^2 (\lambda_{4,0}-\lambda_{3,0})} {(\lambda_{4,0}-\lambda_{3,0})u_1  + (\lambda_{3,0}-\lambda_{1,0}) u_4}  \frac{1}{u_1 u_4^2} -\omega^2. 
\end{equation}
The condition $p=q$ implies that in a neighborhood of $t_B$,
\begin{equation*}
\lim_{t\to t_B^-} (u_1u_4^2)'(t) > 0,
\end{equation*}
as mentioned in \eqref{temp91}. Thus, $1/({u_1 u_4^2})$ is a decreasing function near $t_B$ and converges. Moreover, in the case of \eqref{t_case2} we have that   
\begin{equation*}
\frac{-p^2 (\lambda_{4,0}-\lambda_{3,0})} {(\lambda_{4,0}-\lambda_{3,0})u_1  + (\lambda_{3,0}-\lambda_{10}) u_4} \to 0. 
\end{equation*}    
Hence, we conclude that  
\begin{equation*}
\lim_{t\to t_B^-} \delta'(t) = -\omega^2. 
\end{equation*} 
\end{proof}

\begin{proof}[Continue Proof of Theorem \ref{thm_p=q}]
Substituting \eqref{claim2-4} and \eqref{claim2-5} into \eqref{claim2-6} and \eqref{claim2-7} yields \begin{align*}
-\frac{p}{(u_1u_4)(t)} - \eta(t) -\lambda_{1,0} &= \int_0^t \Big(-2\frac{p\eta(s)}{(u_1u_4)(s)} - \eta^2(s)\Big)   ds  + \delta(t),\\
 \frac{p}{(u_1u_4)(t)} -\eta(t)  -\lambda_{4,0} &= \int_0^t \Big(2\frac{p\eta(s)}{(u_1u_4)(s)} - \eta^2(s)\Big)   ds  + \delta(t). 
\end{align*}
We deduce that 
\begin{equation*}
\eta(t)  = \int_0^t \eta^2(s) ds - \delta(t)- \frac{\lambda_{1,0}+\lambda_{4,0}}{2}.  
\end{equation*}
Notice that the integral equation together with \eqref{claim2-10} and Lemma \ref{lemma_delta} yields  
\begin{equation}\label{inf_eta}
\eta(t) \to  \infty, \quad \text{as~} t \to t_B^-.
\end{equation}
The integral equation can be rewritten as 
\begin{equation}\label{eta_de}
\eta' = \eta^2 - \delta', \quad \eta(0) = - \frac{\lambda_{1,0} + \lambda_{4,0}}{2}.
\end{equation}
Then, for $t$ sufficiently close to $t_B$ so that   
\begin{equation*}
 0 < \int_{t}^{t_B} \frac{\eta^2(s) - \delta'(s)}{\eta^2(s) +1} ds< \pi,  
\end{equation*}
we have that
\begin{align*}
\arctan({\eta(\tau)}) -  \arctan({\eta(t)}) = \int_t^{\tau} \frac{\eta^2(s) - \delta'(s)}{\eta^2(s) +1} ds, \quad t < \tau  < t_B.
\end{align*}
Now, send $\tau \to t_B^-$ to obtain   
\begin{align*}
\eta(t) = \cot \left(\int_{t}^{t_B} \frac{\eta^2(s) - \delta'(s)}{\eta^2(s) +1} ds\right). 
\end{align*}   
Owing to Lemma \ref{lemma_delta} and \eqref{inf_eta}, we have 
\begin{equation*}
\lim_{t\to t_B^-} (t_B-t) \eta(t) = 1
\end{equation*}
and for some $\sigma$, 
\begin{equation}\label{eta_form}
\eta(t) = \frac{1}{t_B-t} + \sigma(t), \quad \sigma(t) = o(t_B-t).
\end{equation}
Moreover, one can show that $\sigma$ is integrable, i.e., 
\begin{equation}\label{bounded_sigma}
\Big|\int_0^{t_B} \sigma(s) ds\Big| < \infty.
\end{equation}
Indeed, substituting \eqref{eta_form} into \eqref{eta_de} yields 
\begin{align*}
\sigma'(t) = \sigma^2 (t) + \frac{2\sigma(t)}{t_B-t} - \delta'(t),
\end{align*}
and for $t_1 < t <t_B$
\begin{equation*}
(t_B -t ) \sigma(t) - (t_B - t_1)\sigma(t_1) = \int_{t_1}^t (t_B-s)\sigma^2(s) ds + \int_{t_1}^t \sigma(s) ds - \int_{t_1}^t (t_B-s) \delta'(s) ds. 
\end{equation*}
If $\int_{t_1}^t (t_B-s)\sigma^2(s) ds$ were unbounded, then $\int_{t_1}^t \sigma(s) ds \to -\infty$ as $t\to t_B^-$, as the left-hand side and $\delta'(t)$  converge. However, this yields a contradiction, because  
\begin{equation*}
\frac{\int_{t_1}^t (t_B-s)\sigma^2(s) ds}{\int_{t_1}^t \sigma(s) ds} \to 0.
\end{equation*}    
Thus, for $t_1$ sufficiently close to $t_B$, 
\begin{align*}
\Big| (t_B -t ) \sigma(t) - (t_B - t_1)\sigma(t_1)  \Big| &= \left| \int_{t_1}^t (t_B-s)\sigma^2(s) ds + \int_{t_1}^t \sigma(s) ds - \int_{t_1}^t (t_B-s) \delta'(s) ds \right| \\
& >  \int_{t_1}^t (t_B-s)\sigma^2(s) ds \to \infty, \quad \text{as~} t\to t_B^-.
\end{align*}
Hence,  $\int_{t_1}^t (t_B-s)\sigma^2(s) ds$ converges. Thus we have \eqref{bounded_sigma}.  

\sk We now estimate $\lambda_i$ and $\rho$ from the representation of $\eta$. Integrating  \eqref{claim2-9} together with \eqref{eta_form} yields  
\begin{align*}
(u_1u_4)(t) &= \frac{(t_B-t)^2}{t_B^2 e^{2 \int_0^t \sigma(s) ds }}.   
\end{align*} 
Substituting this representation and \eqref{eta_form} into \eqref{claim2-4},  \eqref{claim2-5}, and \eqref{claim2-3} yields  
\begin{align*}
\lambda_1 (t) &= - \frac{p t_B^2 e^{2 \int_0^t \sigma(s) ds } }{(t_B-t)^2}   - \frac{1}{(t_B-t)} - \sigma(t). \\
\lambda_4 (t) &=   \frac{p t_B^2 e^{2 \int_0^t \sigma(s) ds }}{(t_B-t)^2}   - \frac{1}{(t_B-t)} - \sigma(t), 
 \end{align*} 
and 
\begin{align*}
\rho(t)  = \frac{4 p^2 t_B^4 e^{4 \int_0^t \sigma(s) ds } } {k(t_B-t)^4} + \frac{4\delta'(t)}{k} + c_b. 
\end{align*} 
These representations, together with \eqref{bounded_sigma}, imply (c) of (\rom{2}) in Theorem \ref{mainthm+}. 
\end{proof}

\sk We close this section by providing  a specific example with $p=q$.  

\begin{ex} 
Recall that 
\begin{equation*}
\omega = \sqrt{\frac{kc_b}{4}}, \quad p=q=\frac{\lambda_{4,0} -\lambda_{1,0}}{2}. 
\end{equation*}
Let $\lambda_{3,0} = \lambda_{4,0}$. Then, we have that 
\begin{align*}
k \rho_0 =1, \quad \delta'(t) = -\omega^2,
\end{align*}
from \eqref{k_rho} and \eqref{ex_delta}, respectively. Furthermore, we can obtain an explicit formula for $\eta$ by solving \eqref{eta_de}: 
\begin{equation*}
\eta(t) = \omega \tan \left( \omega  t  - \arctan \left( \frac{\lambda_{1,0}+\lambda_{4,0}}{2\omega}\right)\right). 
\end{equation*}
The maximum interval of existence follows from the domain of $\eta$:
\begin{equation*}
t_B = \frac{\pi/2 + \arctan \left( \frac{\lambda_{1,0}+\lambda_{4,0}}{2\omega} \right) }{\omega}. 
\end{equation*}
Then, integrating \eqref{claim2-9} yields  
\begin{align*}
(u_1u_4)(t) &= \left( \left( \frac{\lambda_{1,0}+\lambda_{4,0}}{2\omega} \right)^2 +1 \right) \cos^2 \left(\omega t - \arctan \left( \frac{\lambda_{1,0}+\lambda_{4,0}}{2\omega} \right)\right).
\end{align*}
Finally, we deduce from \eqref{claim2-4}, \eqref{claim2-5}, and \eqref{claim2-3} that    
\begin{align*}
\lambda_1 &=\lambda_1= - \frac{p}{\left( \frac{\lambda_{1,0}+\lambda_{4,0}}{2\omega} \right)^2 +1 } \sec^2 \left(\omega t - \arctan \left( \frac{\lambda_{1,0}+\lambda_{4,0}}{2\omega} \right)\right) - \omega \tan \left( \omega  t  - \arctan \left( \frac{\lambda_{1,0}+\lambda_{4,0}}{2\omega}\right)\right), \\
\lambda_3 &=\lambda_4 =  \frac{p}{\left( \frac{\lambda_{1,0}+\lambda_{4,0}}{2\omega} \right)^2 +1 } \sec^2 \left(\omega t - \arctan \left( \frac{\lambda_{1,0}+\lambda_{4,0}}{2\omega} \right)\right) - \omega \tan \left( \omega  t  - \arctan \left( \frac{\lambda_{1,0}+\lambda_{4,0}}{2\omega}\right)\right),\\
\rho &=   \frac{\rho_0}{ \left( \left( \frac{\lambda_{1,0}+\lambda_{4,0}}{2\omega} \right)^2 +1\right)^2 } \sec^4 \left(\omega t - \arctan \left( \frac{\lambda_{1,0}+\lambda_{4,0}}{2\omega} \right)\right). 
\end{align*} 
\end{ex}

\section*{Acknowledgments}  Liu was partially supported by the National Science Foundation under Grant DMS1812666. Shin was supported by newly appointed professor research fund of Hanbat National University and the National Research Foundation under Grant NRF-2017R1E1A1A03070498.
\bigskip

\bibliographystyle{plain} 

 \end{document}